\documentclass[a4paper]{amsart}
\usepackage{amsthm,amsfonts,amsmath,amssymb}
\usepackage[abs]{overpic}

\newtheorem{theorem}{Theorem}[section]
\newtheorem{lemma}[theorem]{Lemma}
\newtheorem{sublemma}[theorem]{Sublemma}
\newtheorem{proposition}[theorem]{Proposition}
\newtheorem{corollary}[theorem]{Corollary}
\newtheorem{claim}[theorem]{Claim}

\theoremstyle{remark}
\newtheorem{remark}[theorem]{Remark}
\newtheorem*{acknowledgements}{Acknowledgements}

\theoremstyle{definition}
\newtheorem{definition}[theorem]{Definition}
\newtheorem{example}[theorem]{Example}

\newcommand{\omu}{\overline{\mu}}
\newcommand{\e}{\varepsilon}
\newcommand{\de}{\delta}

\makeatletter

\@addtoreset{figure}{section}
\makeatother

\makeatletter
  
  \@addtoreset{equation}{section}
\makeatother

\begin{document}
\title[Combinatorial approach to Milnor invariants of welded links]{Combinatorial approach to Milnor invariants of welded links}
 
\author[Haruko A. Miyazawa]{Haruko A. Miyazawa}
\address{Institute for Mathematics and Computer Science, Tsuda University,
2-1-1 Tsuda-machi, Kodaira, Tokyo 187-8577, Japan}
\curraddr{}
\email{aida@tsuda.ac.jp}
\thanks{}

\author[Kodai Wada]{Kodai Wada}
\address{Department of Mathematics, Graduate School of Science, Osaka University, 1-1 Machikaneyama-cho, Toyonaka, Osaka 560-0043, Japan}
\curraddr{}
\email{ko-wada@cr.math.sci.osaka-u.ac.jp}
\thanks{The second author was supported by JSPS KAKENHI Grant Number JP19J00006.}

\author[Akira Yasuhara]{Akira Yasuhara}
\address{Faculty of Commerce, Waseda University, 1-6-1 Nishi-Waseda, Shinjuku-ku, Tokyo 169-8050, Japan}
\curraddr{}
\email{yasuhara@waseda.jp}
\thanks{The third author was supported by JSPS KAKENHI Grant Number JP17K05264 and Waseda University Grant for Special Research Projects 
(Project number: 2020C-175, 2020R-018).}



\begin{abstract}
For a classical link, Milnor defined a family of isotopy invariants, called Milnor $\omu$-invariants. 
Recently, Chrisman extended Milnor $\omu$-invariants to welded links by a topological approach. 
The aim of this paper is to show that Milnor $\omu$-invariants can be extended to welded links by a combinatorial approach. 
The proof contains an alternative proof for the invariance of the original $\omu$-invariants of classical links. 
\end{abstract}

\maketitle

\section{Introduction} 
In~\cite{M54,M57}, Milnor defined a family of isotopy invariants of classical links in the $3$-sphere, called {\em Milnor $\omu$-invariants}. 
Given an $n$-component classical link $L$, the {\em Milnor number $\mu_{L}(I)\in\mathbb{Z}$} of $L$ is specified by a finite sequence $I$ of indices in $\{1,\ldots,n\}$. 
This integer is only well-defined up to a certain indeterminacy~$\Delta_L(I)$, 
i.e. the residue class $\omu_{L}(I)$ of $\mu_{L}(I)$ modulo $\Delta_L(I)$ is an invariant of $L$. 
It is shown in~\cite[Theorem~8]{M57} that $\omu_{L}(I)$ is invariant under link-homotopy when the sequence $I$ has no repeated indices. 
Here, {\em link-homotopy} is an equivalence relation generated by self-crossing changes and isotopies (cf.~\cite{M54}). 
In~\cite{HL}, Habegger and Lin defined Milnor numbers for {\em classical string links} in the $3$-ball, and  
proved that they are integer-valued invariants. 
In this sense, Milnor numbers are suitable for classical string links rather than classical links. 
These numbers for classical string links are called {\em Milnor $\mu$-invariants}. 

The notion of virtual links, introduced by Kauffman in~\cite{Kauffman}, is a diagrammatic generalization of classical links in the $3$-sphere. 
It naturally yields the notion of virtual string links. 
{\em Virtual (string) links} are generalized (string) link diagrams considered up to an extended set of Reidemeister moves. 
In the virtual context, there are two other disallowed moves, known as the {\it forbidden moves}. 
{\em Welded (string) links} arise from virtual (string) links when we allow to use one of the two forbidden moves, called the over-crossings commute move. 
The notion of welded objects was first studied by Fenn, Rim\'{a}nyi and Rourke in \cite{FRR}. 
The aim of this paper is to give an extension of Milnor $\omu$-invariants to welded links in a combinatorial way.

In~\cite{DK}, Dye and Kauffman first tried to extend Milnor link-homotopy $\omu$-invariants to virtual links. 
Kotorii pointed out in \cite[Remark~4.6]{Kotorii} that the extension of Dye and Kauffman is incorrect. 
In fact, there exists a classical link having two different values of the Dye-Kauffman's~$\omu$. 
Hence the Dye-Kauffman's~$\omu$ is not well-defined even for classical links (see Remark~\ref{rem-DK-Delta}).

A successful extension is due to Kravchenko and Polyak in~\cite{KP}. 
Using Gauss diagrams, 
they extended Milnor link-homotopy $\mu$-invariants to virtual tangles, which are slight generalizations of virtual string links. 
In~\cite{Kotorii}, Kotorii gave an extension of Milnor link-homotopy $\omu$-invariants to virtual links via the theory of nanowords introduced by Turaev in~\cite{T}. 
Both extensions are actually invariants of welded objects and combinatorial, but they are restricted to the case of link-homotopy invariants. 

In~\cite{ABMW}, Audoux, Bellingeri, Meilhan and Wagner defined a $4$-dimensional version of Milnor 
$\mu$-invariants. Combining this version of Milnor $\mu$-invariants with the Tube map, they extended Milnor isotopy $\mu$-invariants to welded string links. 
Here, the {\em Tube map} is a map from welded string links to ribbon $2$-dimensional string links in the $4$-ball (cf.~\cite{Y,S}). 
Recently, Chrisman in~\cite{C} defined Milnor $\omu$-invariants for welded links with similar ingredients as in \cite{ABMW}, 
and proved that they are {\em welded concordance} invariants. 
While Milnor invariants for welded objects are given in \cite{ABMW,C}, their approaches are topological. 
The authors believe that it is important to consider a combinatorial approach, 
since the advantage of virtual/welded objects is that they are combinatorial.

In~\cite{M57}, Milnor gave an algorithm to compute $\omu$-invariants for a classical link based on its diagram. 
This algorithm can be applied to virtual link diagrams. 
By the result of Chrisman in~\cite{C}, the values given by the algorithm are invariants of welded links. 
Hence, it is theoretically possible to prove that the values are invariant under welded isotopies, from a diagrammatic point of view. 
In this paper, we actually give such a diagrammatic proof. 
Our approach is purely combinatorial, self contained, and different from~\cite{KP,Kotorii,ABMW,C}.

\begin{acknowledgements}
The authors would like to thank Benjamin Audoux and Jean-Baptiste Meilhan for helpful comments on an early version of this paper. 
We are also grateful to Micah Chrisman for informing us of his paper~\cite{C}. 
\end{acknowledgements}

\section{Preliminaries}\label{sec-prelim}
For an integer $n\geq1$, an {\em $n$-component virtual link diagram} is the image of an immersion of $n$ ordered and oriented circles into the plane, whose singularities are only transverse double points. 
Such double points are divided into {\em classical crossings} and {\em virtual crossings} as shown in Figure~\ref{xing}. 

\begin{figure}[htb]
  \begin{center}
    \begin{overpic}[width=4.5cm]{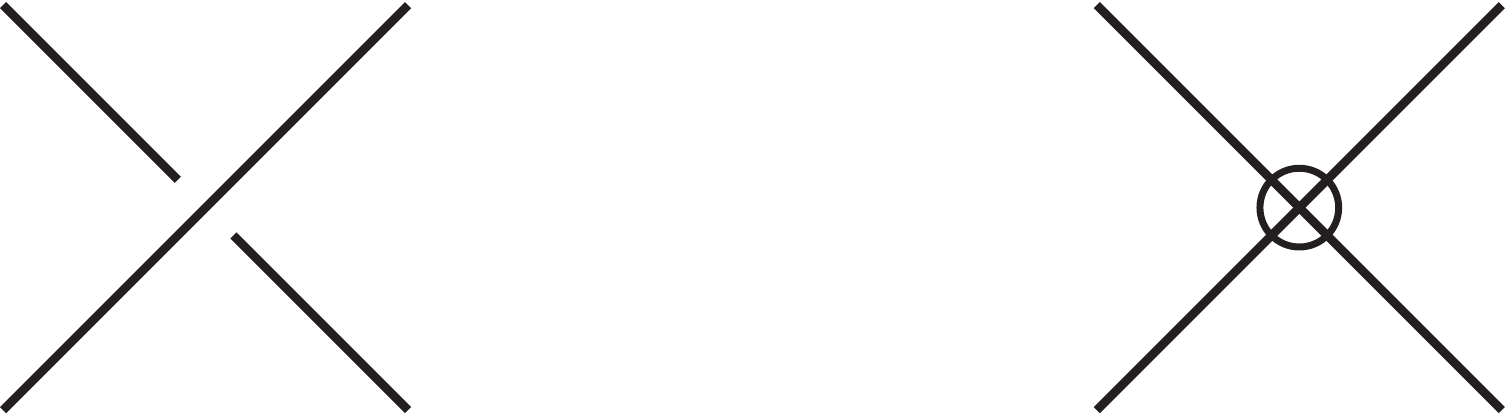}
      \put(-19,-15){classical crossing}
      \put(78,-15){virtual crossing}
    \end{overpic}
  \end{center}
  \vspace{1em}
  \caption{Two types of double points}
  \label{xing}
\end{figure}

{\em Welded Reidemeister moves} consist of Reidemeister moves R1--R3, virtual moves V1--V4 and the over-crossings commute move OC as shown in Figure~\ref{WRmoves}. 
A {\em welded isotopy} is a finite sequence of welded Reidemeister moves, and 
an {\em $n$-component welded link} is an equivalence class of $n$-component virtual link diagrams under welded isotopy. 
We emphasize that all virtual link diagrams and welded links are ordered and oriented. 

\begin{figure}[htb]
  \begin{center}
    \begin{overpic}[width=12cm]{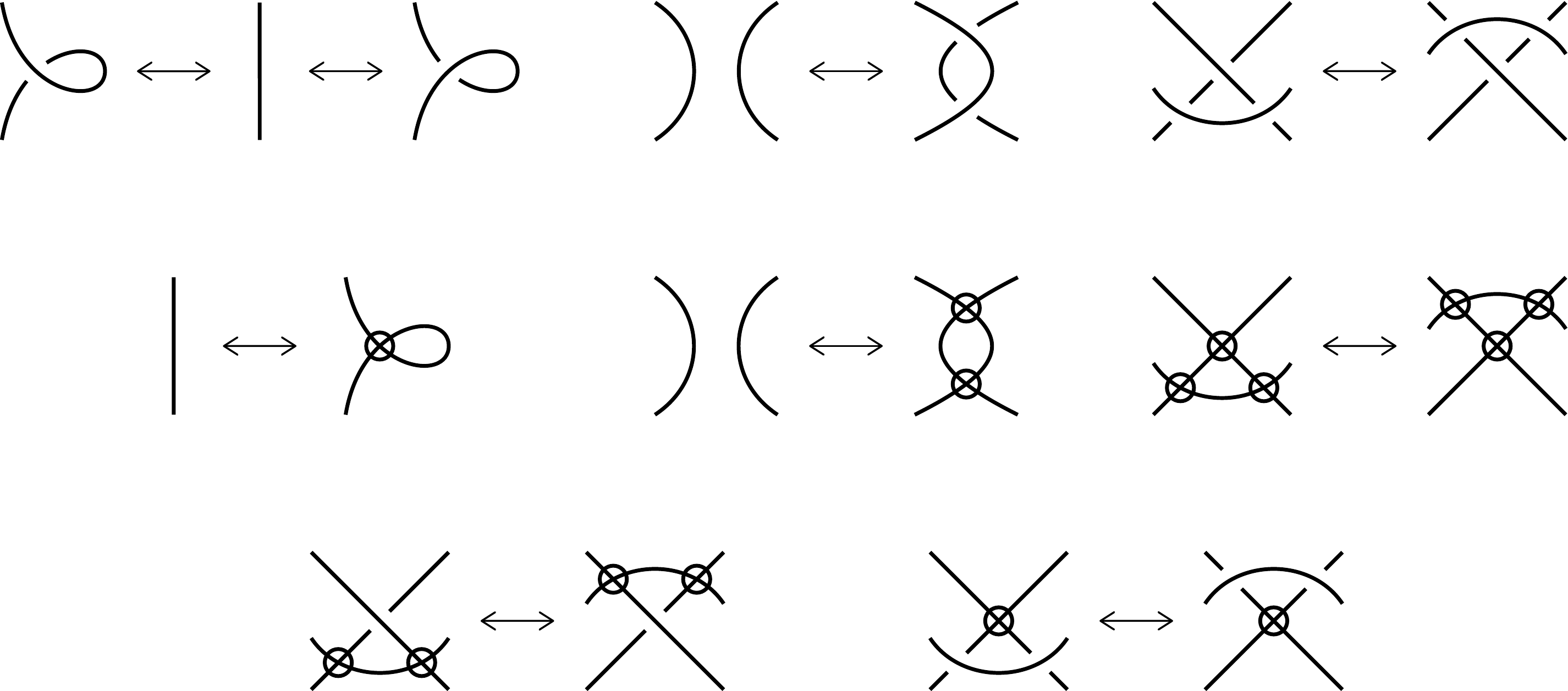}
      \put(32,141){R1}
      \put(69,141){R1}
      \put(178,141){R2}
      \put(290,141){R3}
      \put(50,80.5){V1}
      \put(178,80.5){V2} 
      \put(290,80.5){V3}
      \put(106,21){V4}
      \put(240,21){OC}
    \end{overpic}
  \end{center}
  \caption{Welded Reidemeister moves}
  \label{WRmoves}
\end{figure}
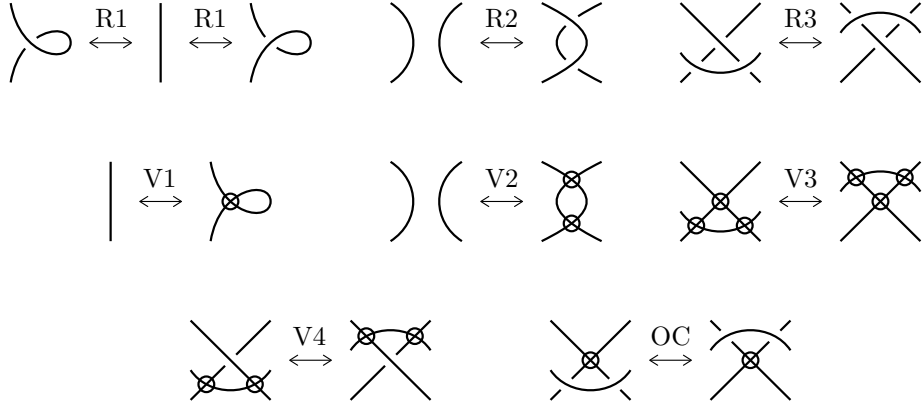

Let $D$ be an $n$-component virtual link diagram. 
Put a {\em base point} $p_{i}$ on some arc of each $i$th component, which is disjoint from all crossings of $D$ $(1\leq i\leq n)$. 
A {\em base point system} of $D$ is an ordered $n$-tuple $\mathbf{p}=(p_{1},\ldots,p_{n})$ of base points on $D$. 
We denote by $(D,\mathbf{p})$ a virtual link diagram $D$ with a base point system $\mathbf{p}$. 
The classical under-crossings of $D$ and base points $p_{1},\ldots,p_{n}$ divide $D$ into a finite number of segments possibly with classical over-crossings and virtual crossings. 
We call such a segment an {\em arc} of $(D,\mathbf{p})$. 

As shown in Figure~\ref{schematic}, let $a_{i1}$ be the outgoing arc from the base point $p_{i}$, and let 
$a_{i2},\ldots,a_{im_{i}+1}$ be the other arcs of the $i$th component in turn with respect to the orientation, where 
$m_i+1$ is the number of arcs of the $i$th component $(1\leq i\leq n)$. 
In the figure, $u_{ij}\in\{a_{kl}\}$ denotes the arc which separates $a_{ij}$ and $a_{ij+1}$. 
Let $\e_{ij}\in\{\pm1\}$ be the sign of the crossing among $a_{ij},u_{ij}$ and $a_{ij+1}$, and we put 
\[
v_{ij}=u_{i1}^{\e_{i1}}u_{i2}^{\e_{i2}}\cdots u_{ij}^{\e_{ij}} 
\]
for $1\leq j\leq m_{i}$. 
We call the word $v_{ij}$ {\em a partial longitude} of $(D,\mathbf{p})$. 

\begin{figure}[htb]
  \begin{center}
    \begin{overpic}[width=11cm]{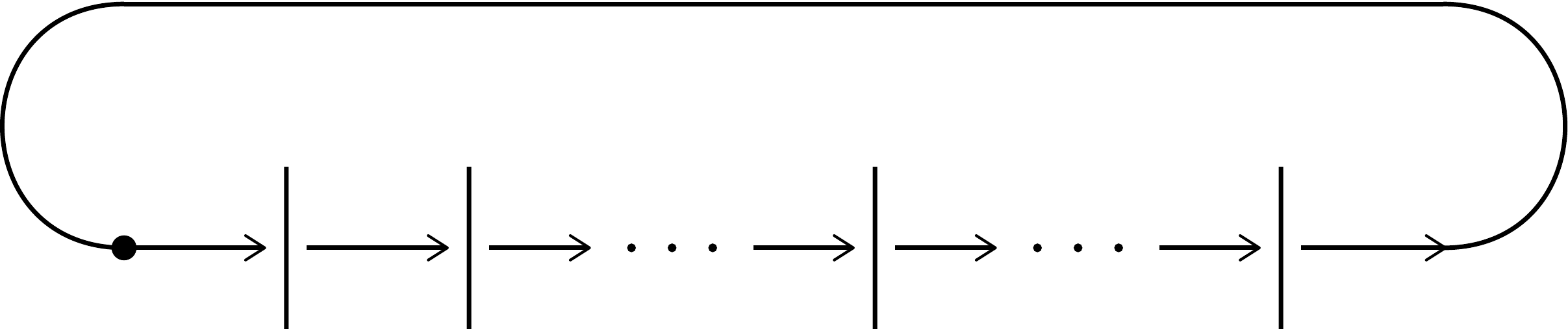}
      \put(21,4){$p_{i}$}
      \put(33,24){$a_{i1}$}
      \put(68,24){$a_{i2}$}
      \put(100,24){$a_{i3}$}
      \put(154,24){$a_{ij}$}
      \put(180,24){$a_{ij+1}$}
      \put(232,24){$a_{im_{i}}$}
      \put(262,24){$a_{im_{i}+1}$}
      \put(52,-12){$u_{i1}$}
      \put(88,-12){$u_{i2}$}
      \put(169,-12){$u_{ij}$}
      \put(250,-12){$u_{im_{i}}$}
    \end{overpic}
  \end{center}
  \vspace{1em}
  \caption{A schematic illustration of the $i$th component}
  \label{schematic}
\end{figure}
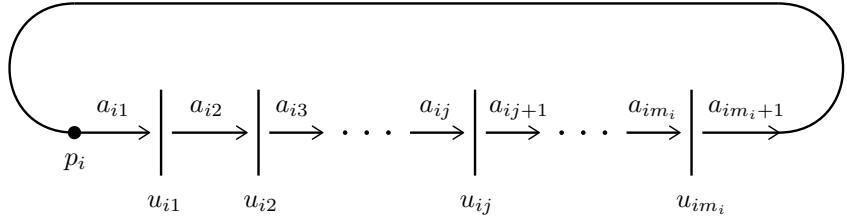

Let $A=\langle\alpha_{1},\ldots,\alpha_{n}\rangle$ be the free group of rank~$n$, 
and let $\overline{A}$ be the free group on the set $\{a_{ij}\}$ of arcs. 
The arcs $a_{ij}$ will be also called {\em letters} 
when they are regarded as elements in $\overline{A}$. 
For an integer $q\geq1$, a sequence of homomorphisms 
\[
\eta_{q}=\eta_{q}(D,\mathbf{p}):\overline{A}\longrightarrow A\] 
associated with $(D,\mathbf{p})$ is defined inductively by 
\begin{eqnarray*}
\eta_{1}(a_{ij})=\alpha_{i}, & \\
\eta_{q+1}(a_{i1})=\alpha_{i} & \mbox{and} \hspace{1em} \eta_{q+1}(a_{ij})=\eta_{q}(v_{ij-1}^{-1})\alpha_{i}\eta_{q}(v_{ij-1}) \hspace{1em} (2\leq j\leq m_i+1). 
\end{eqnarray*} 
Note that our definition of $\eta_q$ is very similar to the original one in~\cite{M57}, 
but they are not the same
because, in \cite{M57}, $a_{i1}\cup a_{im_i+1}$ is a single arc.  
In Section~\ref{sec-w-bar}, we investigate virtual link diagrams with base point systems up to local moves 
relative base point system. The difference of the definition of arcs is essential 
for Theorem~\ref{th-w-bar}, 
see Remark~\ref{rem-Th3}. 

Let $\mathbb{Z}\langle\langle X_{1},\ldots,X_{n}\rangle\rangle$ be the ring of formal power series in non-commutative variables $X_{1},\ldots,X_{n}$ with integer coefficients. 
The {\em Magnus expansion} is a homomorphism 
\[
E:A\longrightarrow \mathbb{Z}\langle\langle X_{1},\ldots,X_{n}\rangle\rangle
\] 
defined, for $1\leq i\leq n$, by 
\[
E(\alpha_{i})=1+X_{i}\hspace{1em}  \mbox{and}\hspace{1em} E(\alpha_{i}^{-1})=1-X_{i}+X_{i}^{2}-X_{i}^{3}+\cdots.
\]

For each $1\leq i\leq n$, let $w_{i}$ be the sum of the signs of all classical self-crossings of the $i$th component. 
We call the word $l_{i}=a_{i1}^{-w_{i}}v_{im_{i}}$ the {\em $i$th preferred longitude} of $(D,\mathbf{p})$. 

\begin{definition}
For a sequence $j_{1}\ldots j_{s}i$ $(1\leq s<q)$ of indices in $\{1,\ldots,n\}$, the {\em Milnor number} $\mu_{(D,\mathbf{p})}^{(q)}(j_{1}\ldots j_{s}i)$ of $(D,\mathbf{p})$ is the coefficient of $X_{j_{1}}\cdots X_{j_{s}}$ in the Magnus expansion $E$ of $\eta_{q}(l_{i})$. 
\end{definition}

The remainder of this section gives several lemmas, which will be used in Section~\ref{sec-w-bar} and subsequent sections. 

Let $q\geq1$ be an integer. 
For a group $G$, we denote by $G_{q}$ the $q$th term of the lower central series of $G$, 
i.e. $G_{1}=G$ and $G_{q+1}=[G,G_{q}]$ is the normal subgroup generated by all $[g,h]=ghg^{-1}h^{-1}$ with $g\in G$ and $h\in G_{q}$. 
For two normal subgroups $N$ and $M$ of $G$, we denote by $NM$ the normal subgroup of $G$ generated by all $nm$ with $n\in N$ and $m\in M$.

The following lemma, which will be used very often in this paper, is easily shown. 

\begin{lemma}[{\cite[page 290]{M57}}]\label{lem-p290}
Let $G$ be a group. 
For any $g,x,y\in G$, the following hold. 
\begin{enumerate}
\item If $x\equiv y\pmod{G_{q}}$, then $x^{-1}gx\equiv y^{-1}gy\pmod{G_{q+1}}$. 

\item Let $N$ be a normal subgroup of $G$. 
If $x\equiv y\pmod{G_{q}N}$, then $x^{-1}gx\equiv y^{-1}gy\pmod{G_{q+1}N}$.
\end{enumerate}
\end{lemma}

For the $q$th term $A_{q}$ of the lower central series of the free group $A=\langle\alpha_{1},\ldots,\alpha_{n}\rangle$, we have the following. 

\begin{lemma}[{\cite[pages 290 and 291]{M57}}]
\label{lem-6k}
The following hold. 
\begin{enumerate}
\item For any $a_{ij}\in\overline{A}$, $\eta_{q}(a_{ij})\equiv\eta_{q+1}(a_{ij}) \pmod{A_{q}}$. 

\item For any $r_{ij}=a^{-1}_{ij+1}u^{-\varepsilon_{ij}}_{ij}a_{ij}u_{ij}^{\varepsilon_{ij}}$ $(1\leq j\leq m_{i})$, $\eta_{q}(r_{ij})\equiv 1\pmod{A_{q}}$. 
\end{enumerate}
\end{lemma}

Although this is essentially shown in \cite{M57}, we give the proof for the readers' convenience.

\begin{proof}[Proof of Lemma~\ref{lem-6k}] 
(1)~
This is proved by induction on $q$. 
Since $A_{1}=A$, the assertion holds for $q=1$. 
Assume that $q\geq1$. 
For $j=1$, we have $\eta_{q+1}(a_{i1})=\alpha_{i}=\eta_{q+2}(a_{i1})$ by definition. 
For $j\geq2$, 
by the induction hypothesis and Lemma~\ref{lem-p290}(1), we have \begin{eqnarray*}
\eta_{q+1}(a_{ij})
&=&\eta_{q}(v_{ij-1}^{-1})\alpha_{i}\eta_{q}(v_{ij-1}) \\
&\equiv& \eta_{q+1}(v_{ij-1}^{-1})\alpha_{i}\eta_{q+1}(v_{ij-1}) \pmod{A_{q+1}} \\
&=&\eta_{q+2}(a_{ij}). 
\end{eqnarray*}  

(2)~
Put $s_{ij}=a_{ij+1}^{-1}v_{ij}^{-1}a_{i1}v_{ij}$ $(1\leq j\leq m_{i})$. 
Then by (1) and Lemma~\ref{lem-p290}(1), it follows that 
\begin{eqnarray*}
\eta_{q}(s_{ij})&=&\eta_{q}(a_{ij+1}^{-1})\eta_{q}(v_{ij}^{-1})\alpha_{i}\eta_{q}(v_{ij}) \\
&=& \left(\eta_{q-1}(v_{ij}^{-1})\alpha_{i}^{-1}\eta_{q-1}(v_{ij})\right)\eta_{q}(v_{ij}^{-1})\alpha_{i}\eta_{q}(v_{ij}) \\
&\equiv& \eta_{q}(v_{ij}^{-1})\alpha_{i}^{-1}\eta_{q}(v_{ij})\eta_{q}(v_{ij}^{-1})\alpha_{i}\eta_{q}(v_{ij}) \pmod{A_{q}} \\
&=&1.  
\end{eqnarray*} 

For $j=1$, since $r_{i1}=s_{i1}$, we have $\eta_{q}(r_{i1})\equiv1\pmod{A_{q}}$. 
For $j\geq2$, we have $r_{ij}=s_{ij}u_{ij}^{-\varepsilon_{ij}}s_{ij-1}^{-1}u_{ij}^{\varepsilon_{ij}}$. 
This implies that 
\begin{eqnarray*}
\eta_{q}(r_{ij})&=&\eta_{q}(s_{ij})\eta_{q}(u_{ij}^{-\varepsilon_{ij}})\eta_{q}(s_{ij-1}^{-1})\eta_{q}(u_{ij}^{\varepsilon_{ij}}) \\
&\equiv& \eta_{q}(u_{ij}^{-\varepsilon_{ij}})\eta_{q}(u_{ij}^{\varepsilon_{ij}}) \pmod{A_{q}}\\
&=&1. 
\end{eqnarray*}
\end{proof}

It is shown in~\cite{MKS} that the Magnus expansion $E$ is injective and satisfies the following. 

\begin{lemma}[{\cite[Corollary 5.7]{MKS}}]
\label{lem-Magnus}
For any $x\in A_{q}$, 
\[
E(x)=1+(\text{terms of degree~$\geq q$}). 
\]
\end{lemma}

By Lemmas~\ref{lem-6k}(1) and \ref{lem-Magnus}, we have the following lemma.

\begin{lemma}\label{lem-length}
If $1\leq s<q$, then $\mu_{(D,\mathbf{p})}^{(q)}(j_{1}\ldots j_{s}i)=\mu_{(D,\mathbf{p})}^{(q+1)}(j_{1}\ldots j_{s}i)$. 
\end{lemma}

Taking the integer $q$ sufficiently large, by this lemma 
we may ignore $q$ and denote $\mu_{(D,\mathbf{p})}^{(q)}(j_{1}\ldots j_{s}i)$ by $\mu_{(D,\mathbf{p})}(j_{1}\ldots j_{s}i)$. 
In the rest of this paper, $q$ is assumed to be a sufficiently large integer.

\section{Milnor numbers and welded isotopy relative base point system}
\label{sec-w-bar}
A {\em local move relative base point system} is a local move on a virtual link diagram with a base point system such that it keeps the positions of base points. 
A {\em $\overline{\mbox{w}}$-isotopy} is a finite sequence of 
welded Reidemeister moves relative base point system and a local move as shown in Figure~\ref{w-bar}. 
We emphasize that in a $\overline{\mbox{w}}$-isotopy, we do {\em not} allow to use two local moves as shown in Figure~\ref{bpt-change}. 
We call the two local moves {\em base-change moves}.

\begin{figure}[htb]
  \begin{center}
    \begin{overpic}[width=6cm]{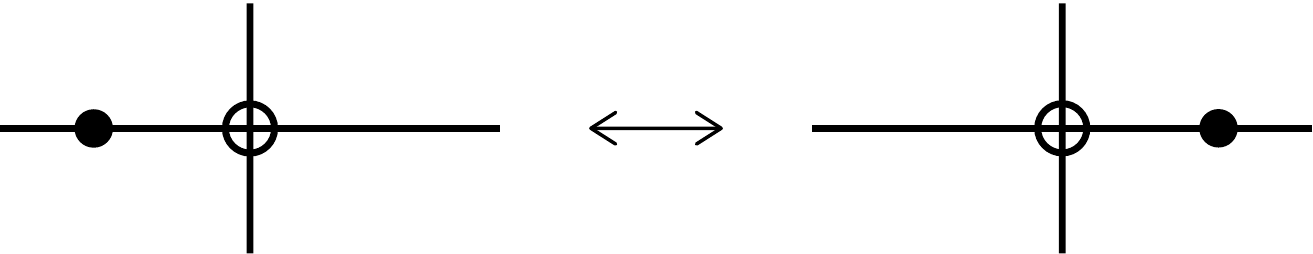}
    \end{overpic}
  \end{center}
  \caption{A base point passing through a virtual crossing}
  \label{w-bar}
\end{figure}

\begin{figure}[htb]
  \begin{center}
    \begin{overpic}[width=12cm]{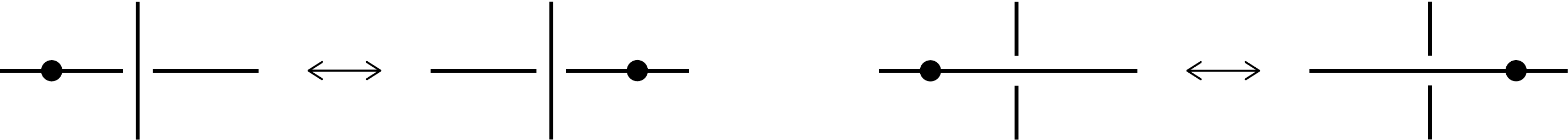}
    \end{overpic}
  \end{center}
  \caption{Base-change moves}
  \label{bpt-change}
\end{figure}

The following theorem gives the invariance of 
Milnor numbers under $\overline{\mbox{w}}$-isotopy.

\begin{theorem}\label{th-w-bar}
Let $(D,\mathbf{p})$ and $(D',\mathbf{p}')$ be virtual link diagrams with base point systems. 
If $(D,\mathbf{p})$ and $(D',\mathbf{p}')$ are $\overline{\mbox{w}}$-isotopic, 
then $\mu_{(D,\mathbf{p})}(I)=\mu_{(D',\mathbf{p}')}(I)$ for any sequence~$I$. 
\end{theorem}

Let $l_{i}$ and $l'_{i}$ be the $i$th preferred longitudes of $(D,\mathbf{p})$ and $(D',\mathbf{p}')$, respectively $(1\leq i\leq n)$.
To show Theorem~\ref{th-w-bar}, we observe the difference between 
$\eta_{q}(D,\mathbf{p})(l_{i})$ and $\eta_{q}(D',\mathbf{p}')(l'_{i})$ under $\overline{\mbox{w}}$-isotopy. 

\begin{proposition}\label{prop-R1}
If $(D,\mathbf{p})$ and $(D',\mathbf{p}')$ are related by a single R1 move relative base point system, then $\eta_{q}(D,\mathbf{p})(l_{i})\equiv\eta_{q}(D',\mathbf{p}')(l'_{i})\pmod{A_{q}}$. 
\end{proposition}

\begin{proof}
There are four R1 moves depending on orientations of strands. 
By~\cite[Theorem~1.1]{P}, it is enough to consider the two moves R1a and R1b in a disk~$\delta$ as shown in Figure~\ref{R1}. 
Here, the symbol~$\circ$ in the figure denotes either a base point or an under-crossing. 
Without loss of generality, we may assume that the R1a/R1b move is applied to the 1st component. 

\begin{figure}[htb]
  \begin{center}
    \vspace{1em}
    \begin{overpic}[width=12cm]{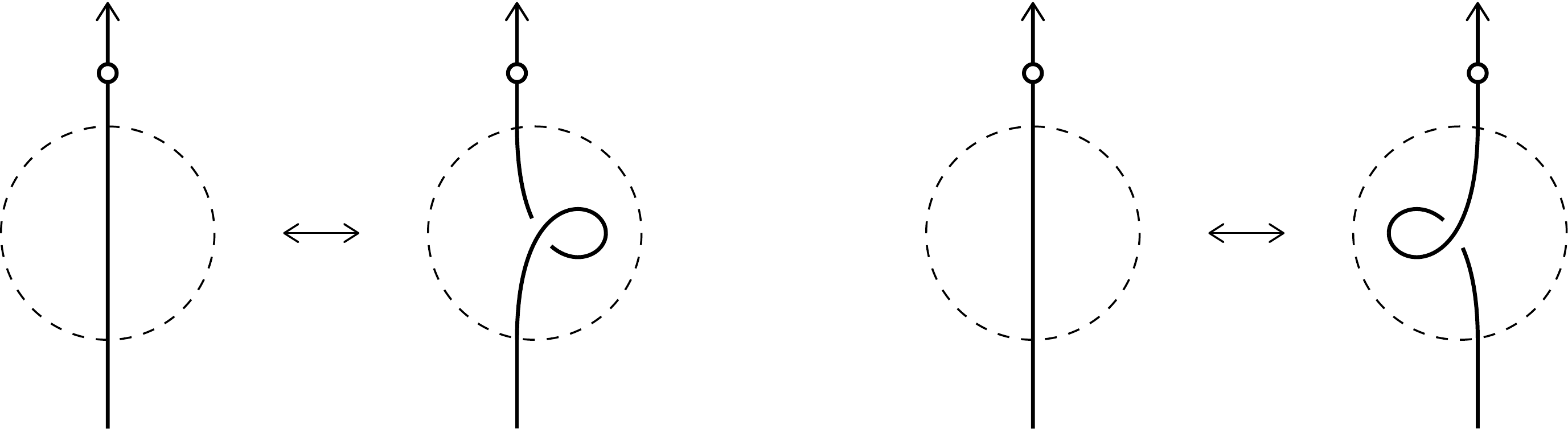}
      \put(61.5,48){R1a}
      \put(9,-15){$(D,\mathbf{p})$}
      \put(96,-15){$(D',\mathbf{p}')$}
      \put(0,61){$\delta$}
      \put(93,61){$\delta$}
      \put(30,5){$a_{1h}$}
      \put(30,95){$a_{1h+1}$}
      \put(30,82){or $a_{11}$}
      \put(119,5){$a'_{1h}$}
      \put(119,95){$a'_{1h+1}$}
      \put(119,82){or $a'_{11}$}
      \put(118,54){$b$}
      \put(262.5,48){R1b}
      \put(210.5,-15){$(D,\mathbf{p})$}
      \put(305,-15){$(D',\mathbf{p}')$}
      \put(202,61){$\delta$}
      \put(294,61){$\delta$}
      \put(230,5){$a_{1h}$}
      \put(230,95){$a_{1h+1}$}
      \put(230,82){or $a_{11}$}
      \put(327,5){$a'_{1h}$}
      \put(327,95){$a'_{1h+1}$}
      \put(327,82){or $a'_{11}$}
      \put(313,54){$b$}
    \end{overpic}
  \end{center}
  \vspace{1em}
  \caption{An R1a/R1b move which relates $(D,\mathbf{p})$ to $(D',\mathbf{p}')$}
  \label{R1}
\end{figure}
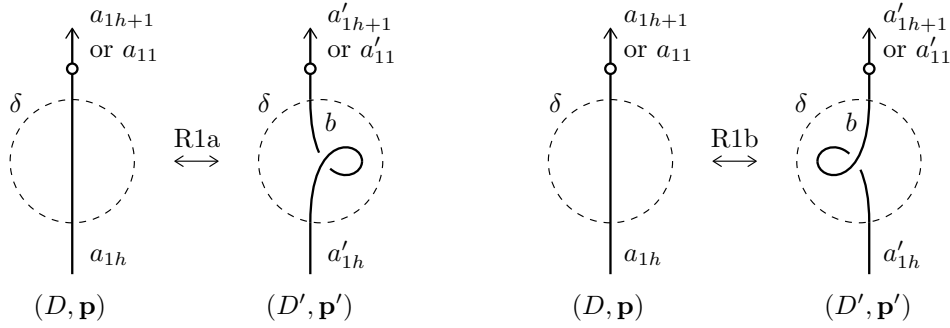 

Let $a_{ij}$ $(1\leq i\leq n, 1\leq j\leq m_i+1)$ be the arcs of $(D,\mathbf{p})$ as given in Section~\ref{sec-prelim}. 
Let $a'_{ij}$ and $b$ be the arcs of $(D',\mathbf{p}')$ such that each $a'_{ij}$ corresponds to the arc $a_{ij}$ of $(D,\mathbf{p})$, and $b$ intersects the disk $\delta$ as shown in Figure~\ref{R1}. 
We put $\eta'_{q}=\eta_{q}(D',\mathbf{p}')$ for short, and note that the domain of $\eta'_{q}$ is the free group $\overline{A'}$ on $\{a'_{ij}\}\cup\{b\}$. 
Since the R1a/R1b move that relates $(D,\mathbf{p})$ to $(D',\mathbf{p}')$ is applied to the $1$st component, $l_{1}$ and $l'_{1}$ can be written in the forms 
\[
l_{1}=a_{11}^{-w_{1}}xy\hspace{1em} \mbox{and}\hspace{1em}  
l'_{1}=
\begin{cases}
(a'_{11})^{-w_{1}-1}x'a'_{1h}y' & \mbox{(R1a move case)}, \\
(a'_{11})^{-w_{1}-1}x'by' &\mbox{(R1b move case)}
\end{cases}
\]
for certain words $x,y\in \overline{A}$ and $x',y'\in\overline{A'}$ such that $x$ and $y$ are obtained from $x'$ and $y'$, respectively, by replacing $b$ with $a_{1h}$, and $a'_{st}$ with $a_{st}$ for all $s,t$. 
For $2\leq i\leq n$, each $l_{i}$ is obtained from $l'_{i}$ by replacing $b$ with $a_{1h}$, and $a'_{st}$ with $a_{st}$ for all $s,t$.
To complete the proof, we need the following.

\begin{claim}\label{claim-R1}
$(1)$~$\eta'_{q}(a'_{1h})\equiv\eta'_{q}(b)\pmod{A_{q}}$. 

$(2)$~$\eta_{q}(a_{ij})\equiv\eta'_{q}(a'_{ij})\pmod{A_{q}}$ for any $i, j$. 
\end{claim} 

Before showing this claim, we observe that it implies Proposition~\ref{prop-R1}. 

For $2\leq i\leq n$, the congruence $\eta_{q}(l_{i})\equiv\eta'_{q}(l'_{i})\pmod{A_{q}}$ follows from Claim~\ref{claim-R1} immediately. 
For $i=1$, in the R1a move case, it follows from Lemma~\ref{lem-6k}(1) and Claim~\ref{claim-R1} that
\begin{eqnarray*}
\eta_{q}(l_{1})=\alpha_{1}^{-w_{1}}\eta_{q}(xy) 
&=& \alpha_{1}^{-w_{1}}\alpha_{1}^{-1}\eta_{q}(x)\left(\eta_{q}(x^{-1})\alpha_{1}\eta_{q}(x)\right)\eta_{q}(y) \\
&=& \alpha_{1}^{-w_{1}-1}\eta_{q}(x)\eta_{q+1}(a_{1h})\eta_{q}(y) \\
&\equiv& \alpha_{1}^{-w_{1}-1}\eta_{q}(x)\eta_{q}(a_{1h})\eta_{q}(y) \pmod{A_{q}} \\ 
&\equiv& \alpha_{1}^{-w_{1}-1}\eta'_{q}(x'a'_{1h}y') \pmod{A_{q}} \\
&=& \eta'_{q}((a'_{11})^{-w_{1}-1}x'a'_{1h}y') =\eta'_{q}(l'_{1}). 
\end{eqnarray*} 
In the R1b move case, we similarly obtain the conclusion. 
\end{proof}

\begin{proof}[{Proof of Claim~\ref{claim-R1}}]
(1)~
By Lemma~\ref{lem-6k}(2), we have 
\[
\begin{cases}
\eta'_{q}(b^{-1}(a'_{1h})^{-1}a'_{1h}a'_{1h})\equiv1\pmod{A_{q}} & \mbox{(R1a move case)}, \\
\eta'_{q}(b^{-1}{b}^{-1}a'_{1h}b)\equiv1\pmod{A_{q}} & \mbox{(R1b move case)}.  
\end{cases}
\]
Hence it follows that 
\begin{eqnarray*}
\eta'_{q}(b^{-1}a'_{1h})=\eta'_{q}(b^{-1}(a'_{1h})^{-1}a'_{1h}a'_{1h}) 
\equiv 1\pmod{A_{q}} 
\end{eqnarray*} 
in the R1a move case, and that 
\begin{eqnarray*}
\eta'_{q}(b^{-1}a'_{1h})=\eta'_{q}(bb^{-1}b^{-1}a'_{1h}bb^{-1}) 
&\equiv& \eta'_{q}(bb^{-1})\pmod{A_{q}} \\
&=& 1 
\end{eqnarray*} 
in the R1b move case. 

(2)~
This is proved by induction on $q$. 
The assertion is certainly true for $q=1$ or $j=1$. 
Assume that $q\geq1$ and $j\geq2$, and consider the R1a move case. 

Let $v_{ij-1}$ and $v'_{ij-1}$ be partial longitudes of $(D,\mathbf{p})$ and $(D',\mathbf{p}')$ as given in Section~\ref{sec-prelim}, respectively. 
If $i\neq1$ or $v'_{ij-1}$ does not contain the word $x'a'_{1h}$, then 
$v_{ij-1}$ is obtained from $v'_{ij-1}$ by replacing $b$ with $a_{1h}$, and $a'_{st}$ with $a_{st}$ for all $s,t$. 
By (1) and the induction hypothesis, we have 
\[
\eta_{q}(v_{ij-1})\equiv\eta'_{q}(v'_{ij-1}) \pmod{A_{q}}. 
\] 
Therefore Lemma~\ref{lem-p290}(1) implies that 
\begin{eqnarray*}
\eta_{q+1}(a_{ij})&=&\eta_{q}(v_{ij-1}^{-1})\alpha_{i}\eta_{q}(v_{ij-1}) \\ 
&\equiv& \eta'_{q}((v'_{ij-1})^{-1})\alpha_{i}\eta'_{q}(v'_{ij-1}) \pmod{A_{q+1}} \\ 
&=&\eta'_{q+1}(a'_{ij}). 
\end{eqnarray*} 
If $i=1$ and $v'_{ij-1}$ contains the word $x'a'_{1h}$, then 
$v_{1j-1}$ and $v'_{1j-1}$ have the forms 
\[
v_{1j-1}=xz\hspace{1em} \mbox{and}\hspace{1em} v'_{1j-1}=x'a'_{1h}z' 
\] 
for certain words $z\in\overline{A}$ and $z'\in\overline{A'}$ such that $z$ is obtained from $z'$ by replacing $b$ with $a_{1h}$, and $a'_{st}$ with $a_{st}$ for all $s,t$. 
Combining (1), Lemma~\ref{lem-6k}(1) and the induction hypothesis, it follows that 
\begin{eqnarray*}
\eta_{q}(v_{1j-1})&=& \alpha_{1}^{-1}\eta_{q}(x)\left(\eta_{q}(x^{-1})\alpha_{1}\eta_{q}(x)\right)\eta_{q}(z) \\ 
&=& \alpha_{1}^{-1}\eta_{q}(x)\eta_{q+1}(a_{1h})\eta_{q}(z) \\
&\equiv& \alpha_{1}^{-1}\eta_{q}(x)\eta_{q}(a_{1h})\eta_{q}(z) \pmod{A_{q}} \\
&\equiv& \alpha_{1}^{-1}\eta'_{q}(x'a'_{1h}z')\pmod{A_{q}} \\
&=& \alpha_{1}^{-1}\eta'_{q}(v'_{ij-1}). 
\end{eqnarray*}
Hence Lemma~\ref{lem-p290}(1) implies that 
\begin{eqnarray*}
\eta_{q+1}(a_{1j})&=&\eta_{q}(v_{1j-1}^{-1})\alpha_{1}\eta_{q}(v_{1j-1}) \\ 
&\equiv& \eta'_{q}((v'_{1j-1})^{-1})\alpha_{1}\alpha_{1}\alpha_{1}^{-1}\eta'_{q}(v'_{ij-1})\pmod{A_{q+1}} \\
&=&\eta_{q+1}(a'_{1j}). 
\end{eqnarray*} 
The proof for the R1b move case is similar. 
\end{proof}

\begin{proposition}\label{prop-R2}
If $(D,\mathbf{p})$ and $(D',\mathbf{p}')$ are related by a single R2 move relative base point system, then $\eta_{q}(D,\mathbf{p})(l_{i})\equiv\eta_{q}(D',\mathbf{p}')(l'_{i})\pmod{A_{q}}$. 
\end{proposition}

\begin{proof}
There are four R2 moves depending on orientations of strands. 
By~\cite[Theorem~1.1]{P}, it is enough to consider the R2 move in a disk~$\delta$ as shown in Figure~\ref{R2}. 

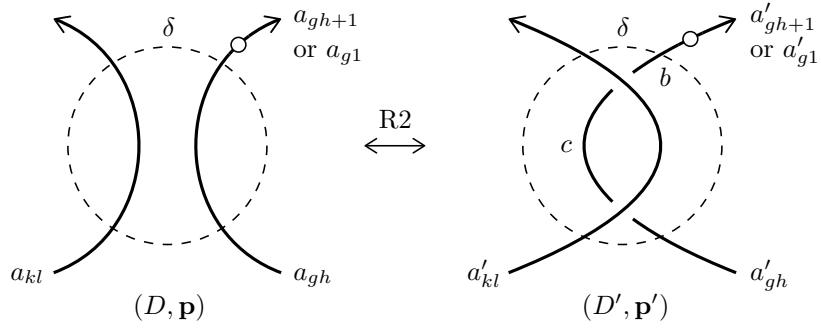
\begin{figure}[htb]
  \begin{center}
    \begin{overpic}[width=9cm]{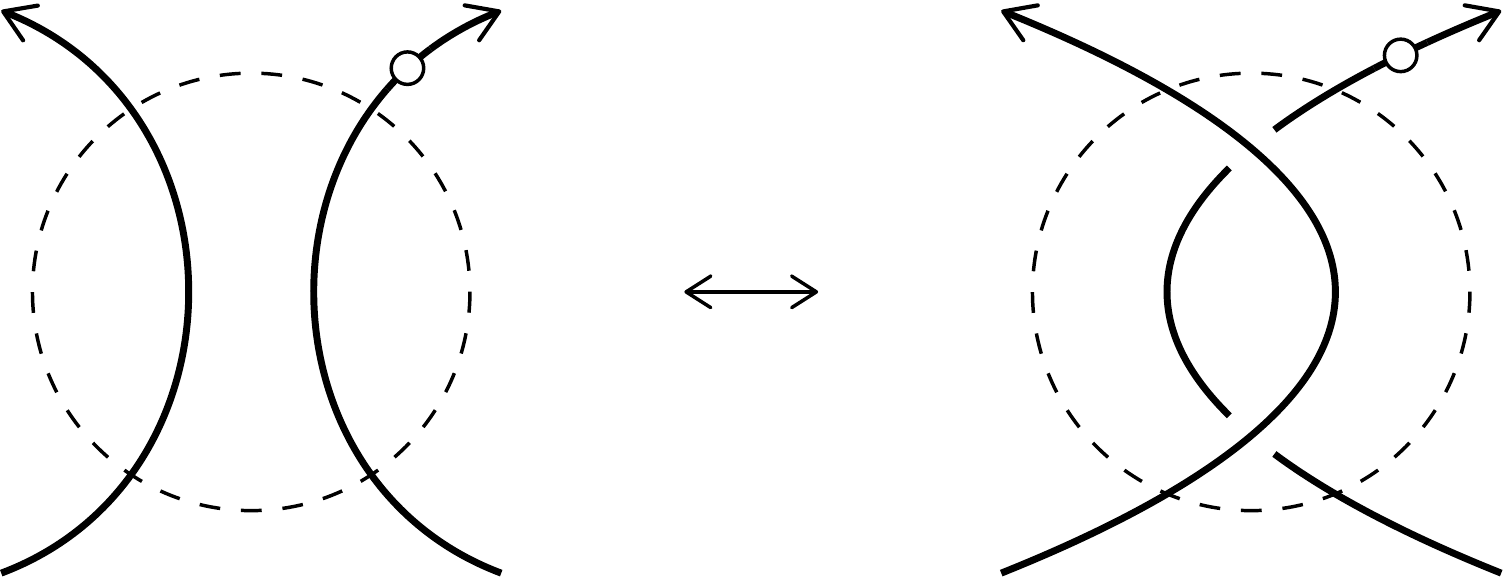}
      \put(122,55){R2}
      \put(30,-15){$(D,\mathbf{p})$}
      \put(41,90){$\delta$}
      \put(-16,-2){$a_{kl}$}
      \put(90,-2){$a_{gh}$}
      \put(90,95){$a_{gh+1}$}
      \put(90,82){or $a_{g1}$}
      \put(198,-15){$(D',\mathbf{p}')$}
      \put(211,90){$\delta$}
      \put(190,46){$c$}
      \put(227,71){$b$}
      \put(155,-2){$a'_{kl}$}
      \put(261,-2){$a'_{gh}$}
      \put(261,95){$a'_{gh+1}$}
      \put(261,82){or $a'_{g1}$}
    \end{overpic}
  \end{center}
  \vspace{1em}
  \caption{An R2 move which relates $(D,\mathbf{p})$ to $(D',\mathbf{p}')$}
  \label{R2}
\end{figure}

Let $a_{ij}$ $(1\leq i\leq n, 1\leq j\leq m_i+1)$ be the arcs of $(D,\mathbf{p})$ as given in Section~\ref{sec-prelim}. 
Let $a'_{ij},b$ and $c$ be the arcs of $(D',\mathbf{p}')$ such that each $a'_{ij}$ corresponds to the arc $a_{ij}$ of $(D,\mathbf{p})$, and $b$ and $c$ intersect the disk $\delta$ as shown in Figure~\ref{R2}. 
Then the domain of $\eta'_{q}=\eta_{q}(D',\mathbf{p}')$ is the free group on $\{a'_{ij}\}\cup\{b,c\}$. 
Note that all the preferred longitudes $l_{i}'$ of $(D',\mathbf{p}')$ do not contain the letter $c$. 
Each $l_{i}$ $(1\leq i\leq n)$ is obtained from $l'_{i}$ by replacing $b$ with $a_{gh}$ in Figure~\ref{R2}, and $a'_{st}$ with $a_{st}$ for all $s,t$. 
Hence Proposition~\ref{prop-R2} follows from the claim below. 
\end{proof}

\begin{claim}\label{claim-R2}
$(1)$~ 
$\eta'_{q}(a'_{gh})\equiv\eta'_{q}(b)\pmod{A_{q}}$. 

$(2)$~ 
$\eta_{q}(a_{ij})\equiv\eta'_{q}(a'_{ij})\pmod{A_{q}}$ for any $i, j$. 
\end{claim}

\begin{proof}
(1)~
By Lemma~\ref{lem-6k}(2), we have 
\[
\eta'_{q}(b^{-1}a'_{kl}\,c\,(a'_{kl})^{-1})\equiv1\pmod{A_{q}}
\]
and 
\[
\eta'_{q}(c^{-1}(a'_{kl})^{-1}a'_{gh}a'_{kl})\equiv1\pmod{A_{q}}. 
\]
This implies that 
\begin{eqnarray*}
\eta'_{q}(b^{-1}a'_{gh})
&=&\eta'_{q}(b^{-1}a'_{kl}\,c\,(a'_{kl})^{-1}a'_{kl}c^{-1}(a'_{kl})^{-1}a'_{gh}a'_{kl}(a'_{kl})^{-1}(a'_{gh})^{-1}a'_{gh}) \\
&\equiv&\eta'_{q}(a'_{kl}(a'_{kl})^{-1}(a'_{gh})^{-1}a'_{gh})\pmod{A_{q}} \\
&=&1. 
\end{eqnarray*}

(2)~This is proved by induction on $q$. 
The assertion certainly holds for $q=1$ or $j=1$. 
Assume that $q\geq1$ and $j\geq2$. 
Let $v_{ij-1}$ and $v'_{ij-1}$ be partial longitudes of $(D,\mathbf{p})$ and $(D',\mathbf{p}')$, respectively. 
Then $v_{ij-1}$ is obtained from $v'_{ij-1}$ by replacing $b$ with $a_{gh}$ in Figure~\ref{R2}, and $a'_{st}$ with $a_{st}$ for all $s,t$. 
By (1) and the induction hypothesis, we have
\[
\eta_{q}(v_{ij-1})\equiv\eta'_{q}(v'_{ij-1}) \pmod{A_{q}}. 
\]
Hence it follows from Lemma~\ref{lem-p290}(1) that 
\begin{eqnarray*}
\eta_{q+1}(a_{ij})&=&\eta_{q}(v_{ij-1}^{-1})\alpha_{i}\eta_{q}(v_{ij-1}) \\
&\equiv& \eta'_{q}((v'_{ij-1})^{-1})\alpha_{i}\eta'_{q}(v'_{ij-1}) \pmod{A_{q+1}} \\
&=& \eta'_{q+1}(a'_{ij}). 
\end{eqnarray*} 
\end{proof}

\begin{proposition}\label{prop-R3}
If $(D,\mathbf{p})$ and $(D',\mathbf{p}')$ are related by a single R3 move relative base point system, then $\eta_{q}(D,\mathbf{p})(l_{i})\equiv\eta_{q}(D',\mathbf{p}')(l'_{i})\pmod{A_{q}}$. 
\end{proposition}

\begin{proof}
There are eight R3 moves depending on orientations of strands. 
By~\cite[Theorem 1.1]{P}, we may consider the R3 move in a disk~$\delta$ as shown in Figure~\ref{R3}. 
Let $a_{ij}$ and $a'_{ij}$ $(1\leq i\leq n, 1\leq j\leq m_i+1)$ be the arcs of $(D,\mathbf{p})$ and $(D',\mathbf{p}')$, respectively, such that each $a'_{ij}$ corresponds to the arc~$a_{ij}$. 
Then the domain of $\eta'_{q}=\eta_{q}(D',\mathbf{p}')$ is the free group $\overline{A'}$ on $\{a'_{ij}\}$. 

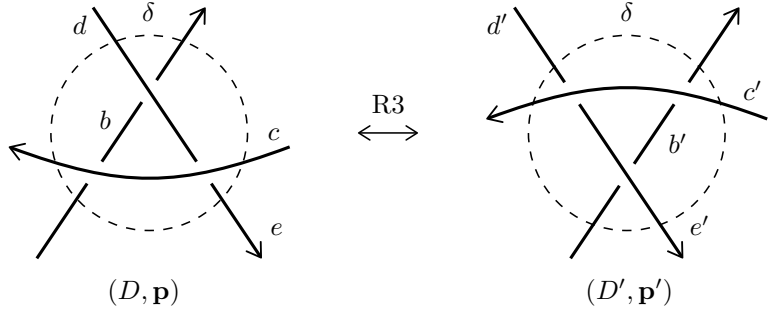
\begin{figure}[htb]
  \begin{center}
    \begin{overpic}[width=10cm]{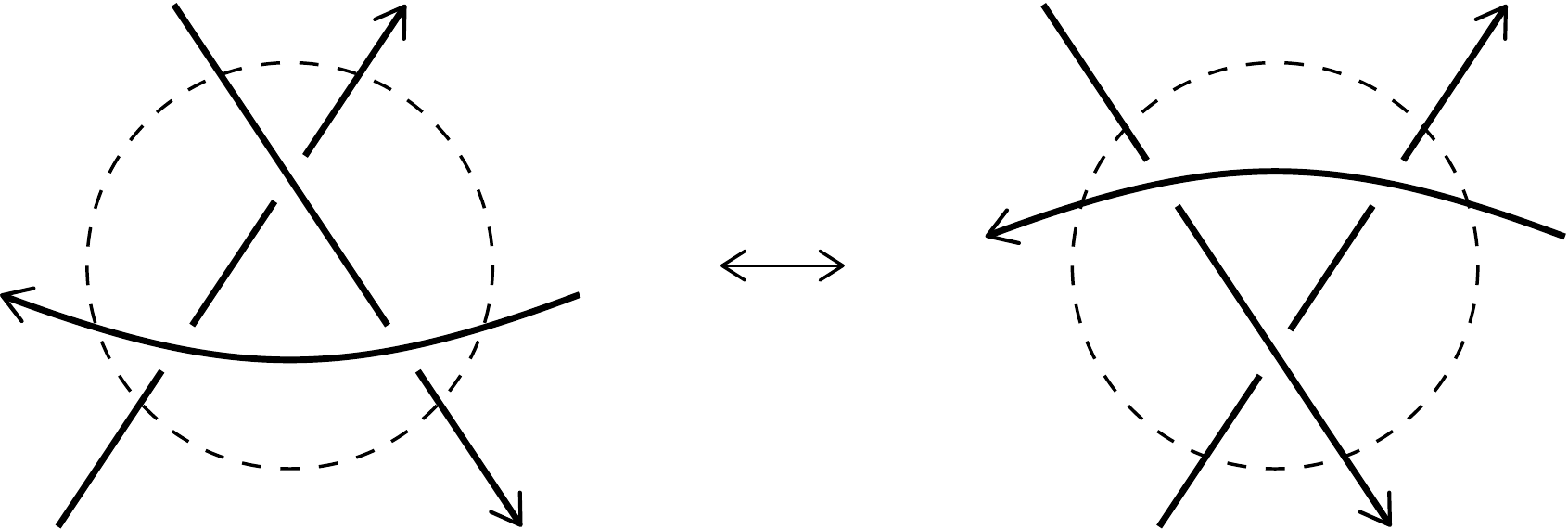}
      \put(136,55){R3}
      \put(37,-15){$(D,\mathbf{p})$}
      \put(50,90){$\delta$}
      \put(34,49){$b$}
      \put(97,45){$c$}
      \put(24,85){$d$}
      \put(98,9){$e$}
      \put(216,-15){$(D',\mathbf{p}')$}
      \put(229,90){$\delta$}
      \put(247,40){$b'$}
      \put(275,60){$c'$}
      \put(179,85){$d'$}
      \put(255,9){$e'$}
    \end{overpic}
  \end{center}
  \vspace{1em}
  \caption{An R3 move which relates $(D,\mathbf{p})$ to $(D',\mathbf{p}')$}
  \label{R3}
\end{figure}

As shown in Figure~\ref{R3}, let $b,c,d,e\in\{a_{kl}\}$ be arcs of $(D,\mathbf{p})$ which intersect~$\delta$.  
Similarly, let $b',c',d',e'\in\{a'_{kl}\}$ be arcs of $(D',\mathbf{p}')$ which intersect~$\delta$. 
Note that $l_{i}$ and $l'_{i}$ do not contain the letters $b$ and $b'$, respectively, for all $1\leq i\leq n$. 
Without loss of generality, we may assume that the arcs $b$ and $b'$ belong to 
the $1$st components of $(D,\mathbf{p})$ and $(D',\mathbf{p}')$, respectively. 
Then $l_{1}$ and $l'_{1}$ can be written in the forms 
\[
l_{1}=a_{11}^{-w_{1}}xc^{-1}dy\hspace{1em} \mbox{and}\hspace{1em}  l'_{1}=(a'_{11})^{-w_{1}}x'e'(c')^{-1}y' 
\] 
for certain words $x,y\in A$ and $x',y'\in\overline{A'}$ such that $x$ and $y$ are obtained from $x'$ and $y'$, respectively, by replacing $a'_{st}$ with $a_{st}$ for all $s,t$. 
For $2\leq i\leq n$, each $l_{i}$ is obtained from $l'_{i}$ by replacing $a'_{st}$ with $a_{st}$ for all $s,t$. 
To complete the proof, we need the following. 

\begin{claim}\label{claim-R3} 
For any letters $a_{ij}\neq b$ of $(D,\mathbf{p})$ and $a'_{ij}\neq b'$ of $(D',\mathbf{p}')$, 
\[\eta_{q}(a_{ij})\equiv\eta'_{q}(a'_{ij}) \pmod{A_{q}}.
\] 
\end{claim}

Before showing this claim, we observe it implies Proposition~\ref{prop-R3}. 

For $2\leq i\leq n$, the congruence $\eta_{q}(l_{i})\equiv\eta'_{q}(l'_{i})\pmod{A_{q}}$ follows from Claim~\ref{claim-R3} immediately. 
By Lemma~\ref{lem-6k}(2), we have $\eta_{q}(e^{-1}c^{-1}dc)\equiv1\pmod{A_{q}}$. 
Hence Claim~\ref{claim-R3} implies that 
\begin{eqnarray*}
\eta_{q}(l_{1})&=&\eta_{q}(a_{11}^{-w_{1}}xc^{-1}dy) \\
&=& \eta_{q}(a_{11}^{-w_{1}}xee^{-1}c^{-1}dcc^{-1}y) \\
&\equiv& \eta_{q}(a_{11}^{-w_{1}}xec^{-1}y) \pmod{A_{q}} \\
&\equiv& \eta'_{q}((a'_{11})^{-w_{1}}x'e'(c')^{-1}y') \pmod{A_{q}} \\
&=& \eta'_{q}(l'_{1}). 
\end{eqnarray*}
\end{proof}

\begin{proof}[{Proof of Claim~\ref{claim-R3}}]
This is proved by induction on $q$. 
The assertion certainly holds for $q=1$ or $j=1$. 
Assume that $q\geq1$ and $j\geq2$. 
Let $v_{ij-1}$ and $v'_{ij-1}$ be partial longitudes of $(D,\mathbf{p})$ and $(D',\mathbf{p}')$, respectively. 
If $i\neq1$ or $v_{ij-1}$ does not contain the word $xc^{-1}d$, 
then it is obtained from $v'_{ij-1}$ by replacing $a'_{st}$ with $a_{st}$ for all $s,t$. 
By the induction hypothesis, we have $\eta_{q}(v_{ij-1})\equiv\eta'_{q}(v'_{ij-1})\pmod{A_{q}}$. 
Therefore Lemma~\ref{lem-p290}(1) implies that 
\[
\eta_{q+1}(a_{ij})\equiv\eta'_{q+1}(a'_{ij})\pmod{A_{q+1}}.  
\] 
If $i=1$ and $v_{ij-1}$ contains $xc^{-1}d$, 
then $v_{1j-1}$ and $v'_{1j-1}$ can be written in the forms 
\[
v_{1j-1}=xc^{-1}dz \hspace{1em} \mbox{and}\hspace{1em} v'_{ij-1}=x'e'(c')^{-1}z'
\]
for certain words $z\in\overline{A}$ and $z'\in\overline{A'}$ such that $z$ is obtained from $z'$ by replacing $a'_{st}$ with $a_{st}$ for all $s,t$. 
Since $\eta_{q}(e^{-1}c^{-1}dc)\equiv1\pmod{A_{q}}$ by Lemma~\ref{lem-6k}(2), it follows that 
\begin{eqnarray*}
\eta_{q}(v_{1j-1})=\eta_{q}(xc^{-1}dz)
\equiv\eta_{q}(xec^{-1}z) \pmod{A_{q}}.
\end{eqnarray*}
By the induction hypothesis we have $\eta_{q}(v_{1j-1})\equiv\eta'_{q}(v'_{1j-1})\pmod{A_{q}}$, 
and hence $\eta_{q+1}(a_{1j})\equiv\eta'_{q+1}(a'_{1j})\pmod{A_{q+1}}$. 
\end{proof}

\begin{proposition}\label{prop-others}
If $(D,\mathbf{p})$ and $(D',\mathbf{p}')$ are related by one of V1--V4, OC relative base point system and the local move in Figure~\ref{w-bar}, then $\eta_{q}(D,\mathbf{p})(l_{i})\equiv\eta_{q}(D',\mathbf{p}')(l'_{i})\pmod{A_{q}}$. 
\end{proposition} 

\begin{proof}
Let $a_{ij}$ $(1\leq i\leq n,1\leq j\leq m_i+1)$ be the arcs of $(D,\mathbf{p})$. 
Then the arcs of $(D',\mathbf{p}')$ can be uniquely determined as shown in Figure~\ref{others}. 
Hence we have $\eta_{q}(D,\mathbf{p})=\eta_{q}(D',\mathbf{p}')$ and $l_{i}=l'_{i}$. 
\end{proof} 

\begin{figure}[htb]
  \begin{center}
    \vspace{1em}
    \begin{overpic}[width=12cm]{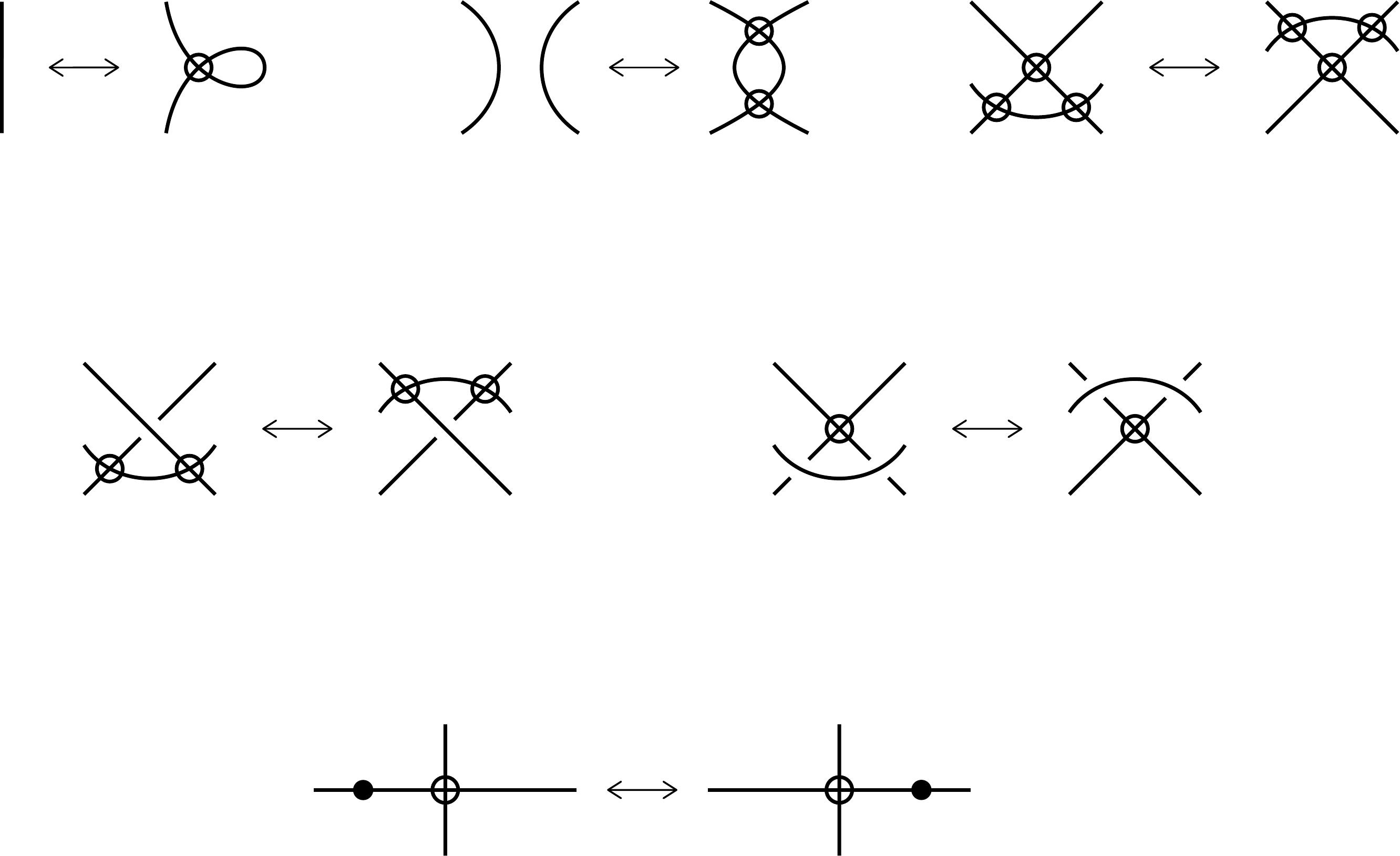}
      \put(15,198){V1}
      \put(151,198){V2}
      \put(283,198){V3}
      \put(67,110){V4}
      \put(233,110){OC} 
      \put(-14,161){$(D,\mathbf{p})$}
      \put(35,161){$(D',\mathbf{p}')$}
      \put(114,161){$(D,\mathbf{p})$}
      \put(168,161){$(D',\mathbf{p}')$}
      \put(239,161){$(D,\mathbf{p})$}
      \put(309,161){$(D',\mathbf{p}')$}
      \put(23,61){$(D,\mathbf{p})$}
      \put(92,61){$(D',\mathbf{p}')$}
      \put(191,61){$(D,\mathbf{p})$}
      \put(261,61){$(D',\mathbf{p}')$}
      \put(94,-15){$(D,\mathbf{p})$}
      \put(188,-15){$(D',\mathbf{p}')$}
      \put(-4,214){$a_{ij}$}
      \put(35,214){$a_{ij}$}
      \put(105,214){$a_{ij}$}
      \put(136,214){$a_{kl}$}
      \put(165,214){$a_{ij}$}
      \put(193,214){$a_{kl}$}
      \put(223,190){$a_{ij}$}
      \put(227,214){$a_{kl}$}
      \put(266,214){$a_{gh}$}
      \put(299,214){$a_{kl}$}
      \put(340,214){$a_{gh}$}
      \put(342,190){$a_{ij}$}
      \put(6,102){$a_{ij}$}
      \put(11,126){$a_{kl}$}
      \put(50,126){$a_{gh}$}
      \put(11,80){$a_{gh'}$}
      \put(83,126){$a_{kl}$}
      \put(122,126){$a_{gh}$}
      \put(127,102){$a_{ij}$}
      \put(83,80){$a_{gh'}$}
      \put(174,102){$a_{ij}$}
      \put(178,126){$a_{kl}$}
      \put(218,126){$a_{gh}$}
      \put(178,80){$a_{gh'}$}
      \put(218,80){$a_{kl'}$}
      
      \put(250,126){$a_{kl}$}
      \put(290,126){$a_{gh}$}
      \put(295,102){$a_{ij}$}
      \put(250,80){$a_{gh'}$}
      \put(290,80){$a_{kl'}$}
      \put(57,23){$a_{im_i+1}$}
      \put(121,23){$a_{i1}$}
      \put(103,37){$a_{kl}$}
      \put(84,5){$p_{i}$}
      \put(172,23){$a_{im_i+1}$}
      \put(230,23){$a_{i1}$}
      \put(199,37){$a_{kl}$}
      \put(221,5){$p_{i}$}
    \end{overpic}
  \end{center}
  \vspace{1em}
  \caption{Proof of Proposition~\ref{prop-others}}
  \label{others}
\end{figure}
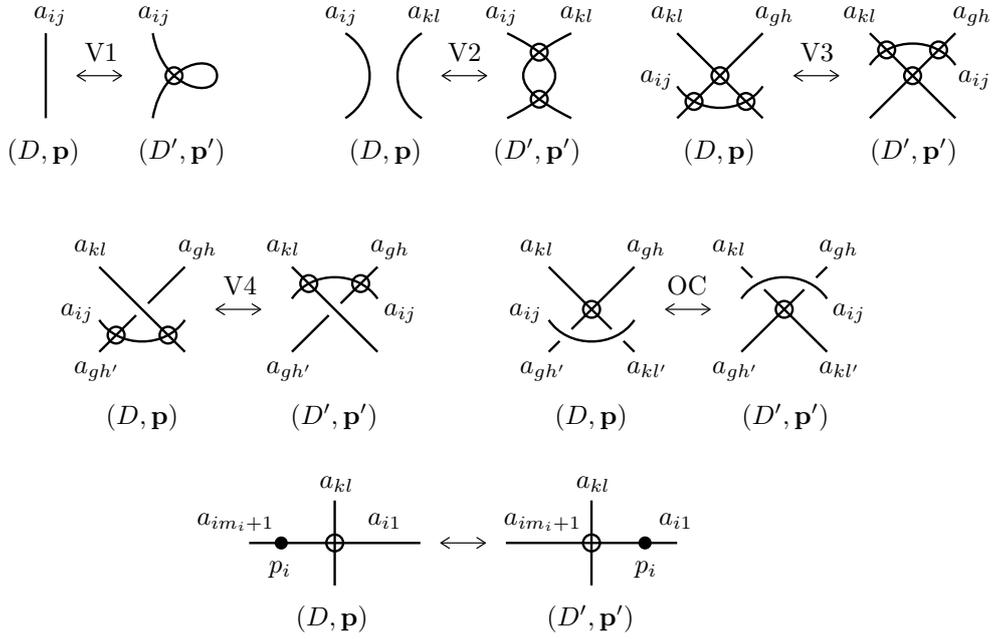

\begin{proof}[{Proof of Theorem~\ref{th-w-bar}}]
By Propositions~\ref{prop-R1}, \ref{prop-R2}, \ref{prop-R3} and \ref{prop-others}, we have 
\[
\eta_{q}(D,\mathbf{p})(l_{i})\equiv \eta_{q}(D',\mathbf{p}')(l'_{i})\pmod{A_{q}}. 
\] 
Then Lemma~\ref{lem-Magnus} implies that 
\[
E(\eta_{q}(D,\mathbf{p})(l_{i}))-E(\eta_{q}(D',\mathbf{p}')(l'_{i}))=(\mbox{terms of degree $\geq q$}).  
\] 
Hence, by definition, $\mu_{(D,\mathbf{p})}(j_{1}\ldots j_{s}i)=\mu_{(D',\mathbf{p}')}(j_{1}\ldots j_{s}i)$ for any sequence $j_{1}\ldots j_{s}i$ with $s<q$. 
\end{proof}

\begin{example}\label{ex-Milnor-number}
Consider the $3$-component virtual link diagram $D$ and its base point system $\mathbf{p}=(p_{1},p_{2},p_{3})$ in the left of Figure~\ref{counter-ex}. 
Let $a_{ij}$ be the arcs of $(D,\mathbf{p})$. 
Since $l_{1}=a_{21}, l_{2}=a_{21}^{-1}(a_{11}a_{23})$ and $ l_{3}=a_{21}^{-1}a_{22}$, by definition we have 
\[
\begin{cases}
\eta_{3}(l_{1})=\alpha_{2}, \\
\eta_{3}(l_{2})=\alpha_{2}^{-1}\alpha_{1}\alpha_{2}^{-1}\alpha_{1}^{-1}\alpha_{2}^{-1}\alpha_{1}\alpha_{2}\alpha_{1}^{-1}\alpha_{2}\alpha_{1}\alpha_{2}^{-1}\alpha_{1}^{-1}\alpha_{2}\alpha_{1}\alpha_{2}, \\ 
\eta_{3}(l_{3})=\alpha_{2}^{-1}\alpha_{1}^{-1}\alpha_{2}\alpha_{1}. 
\end{cases}
\] 
By a direct computation, we have  
\[
\begin{cases}
E(\eta_{3}(l_{1}))=1+X_{2}, \\
E(\eta_{3}(l_{2}))=1+X_{1}+\mbox{(terms of degree~$\geq3$)}, \\
E(\eta_{3}(l_{3}))=1-X_{1}X_{2}+X_{2}X_{1}+\mbox{(terms of degree~$\geq3$)}. 
\end{cases} 
\] 
Hence it follows that 
\[\mu_{(D,\mathbf{p})}(21)=1,\ \mu_{(D,\mathbf{p})}(12)=1,\ \mu_{(D,\mathbf{p})}(123)=-1 \hspace{1em} 
\mbox{and} \hspace{1em} \mu_{(D,\mathbf{p})}(213)=1,
\] 
and that $\mu_{(D,\mathbf{p})}(I)=0$ for any sequence $I$ with length~$\leq3$ except for $21,12,123$ and $213$. 

Consider another base point system $\mathbf{p}'=(p'_{1},p'_{2},p'_{3})$ of $D$ in the right of Figure~\ref{counter-ex}. 
Then we have $l_{1}=a_{22}, l_{2}=a_{21}^{-1}(a_{22}a_{11})$ and $l_{3}=a_{22}^{-1}a_{21}$, 
and hence 
\[
\eta_{3}(l_{1})=\alpha_{2},\ \eta_{3}(l_{2})=\alpha_{1} \hspace{1em} 
\mbox{and} \hspace{1em} \eta_{3}(l_{3})=1. 
\] 
This implies that 
\[
\mu_{(D,\mathbf{p}')}(21)=1 \hspace{1em} 
\mbox{and} \hspace{1em} \mu_{(D,\mathbf{p}')}(12)=1, 
\] 
and that $\mu_{(D,\mathbf{p}')}(I)=0$ for any sequence $I$ with length~$\leq3$ except for $21$ and $12$. 
Therefore, by Theorem~\ref{th-w-bar}, $(D,\mathbf{p})$ and $(D,\mathbf{p'})$ are not $\overline{\rm w}$-isotopic. 
\end{example}

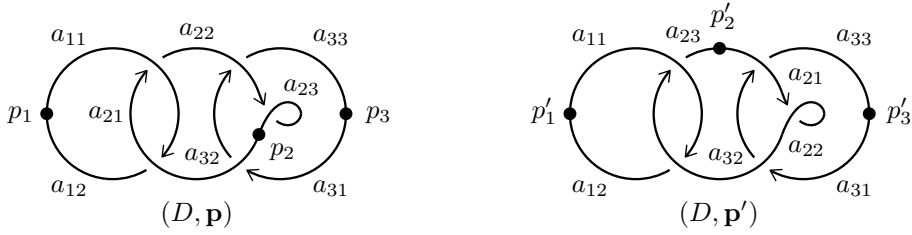
\begin{figure}[htb]
  \begin{center}
    \vspace{1em}
    \begin{overpic}[width=11cm]{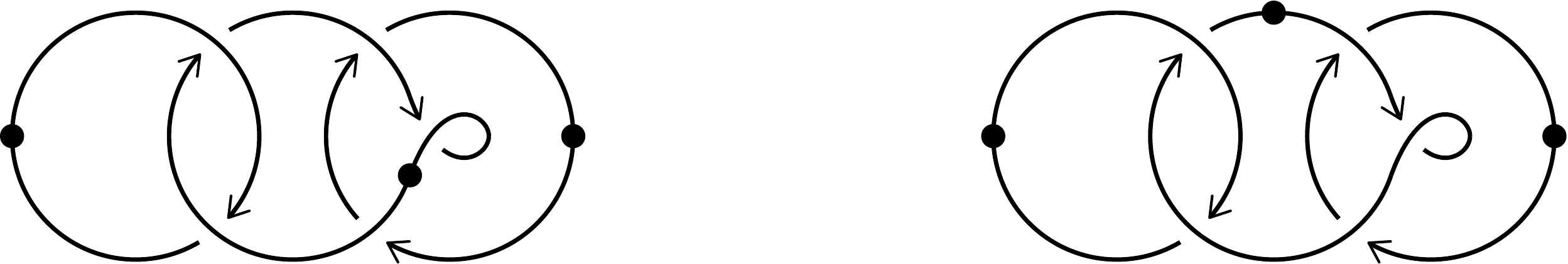}
      \put(45,-15){$(D,\mathbf{p})$}
      \put(-12,24){$p_{1}$}
      \put(86,10){$p_{2}$}
      \put(122,24){$p_{3}$}
      \put(4,53){$a_{11}$}
      \put(4,-4){$a_{12}$}
      \put(18,24){$a_{21}$}
      \put(52,55){$a_{22}$}
      \put(91,34){$a_{23}$}
      \put(102,-4){$a_{31}$}
      \put(54,8){$a_{32}$}
      \put(102,53){$a_{33}$}
      \put(239,-15){$(D,\mathbf{p}')$}
      \put(184,24){$p'_{1}$}
      \put(251,60){$p'_{2}$}
      \put(317,24){$p'_{3}$}
      \put(199,53){$a_{11}$}
      \put(199,-4){$a_{12}$}
      \put(280,38){$a_{21}$}
      \put(280,11){$a_{22}$}
      \put(234,54){$a_{23}$}
      \put(298,-4){$a_{31}$}
      \put(250,8){$a_{32}$}
      \put(298,53){$a_{33}$}
    \end{overpic}
  \end{center}
  \vspace{1em}
  \caption{A $3$-component link diagram $D$ with different base point systems $\mathbf{p}=(p_{1},p_{2},p_{3})$ and $\mathbf{p'}=(p'_{1},p'_{2},p'_{3})$}
  \label{counter-ex}
\end{figure}

\section{Change of a base point system}\label{sec-bpt-change}
In this section, we fix an $n$-component virtual link diagram $D$, and observe behavior of $\eta_{q}(l_{i})$ under a change of a base point system of $D$ (Theorem~\ref{prop-bpt-change}). 

An {\em arc} of $D$ is a segment along $D$ which goes from a classical under-crossing to the next one, where classical over-crossings and virtual crossings are ignored. 
We emphasize the definition of arcs of $D$ is slightly different from that of arcs of $(D,\mathbf{p})$. 
For each $1\leq i\leq n$, we choose one arc of the $i$th component and  denote it by~$a_{i1}$. 
Let $a_{i2},\ldots,a_{im_{i}}$ be the other arcs of the $i$th component in turn with respect to the orientation, where $m_{i}$ denotes the number of arcs of the $i$th component. 
Throughout this section, we fix these arcs $a_{i1},\ldots,a_{im_{i}}$ for $D$. 

Given a base point system $\mathbf{p}=(p_{1},\ldots,p_{n})$ of $D$, 
let $\mathbf{p}(i)$ denote the integer of the second subscript of the arc containing $p_{i}$ $(1\leq i\leq n)$. 
Consider the virtual link diagram $D$ with a base point system $\mathbf{p}=(p_{1},\ldots,p_{n})$. 
For each $i$th component of $(D,\mathbf{p})$, 
the base point $p_{i}$ divides the arc $a_{i\mathbf{p}(i)}$ of $D$ into two arcs. 
We assign the labels $b_{i}^{\mathbf{p}}$ and $a_{i\mathbf{p}(i)}$ to the two arcs of $(D,\mathbf{p})$ as shown in Figure~\ref{assign-b}. 
The labels of the other arcs of $(D,\mathbf{p})$ are the same as those of the corresponding arcs of $D$. 

\begin{figure}[htb]
  \begin{center}
    \begin{overpic}[width=11cm]{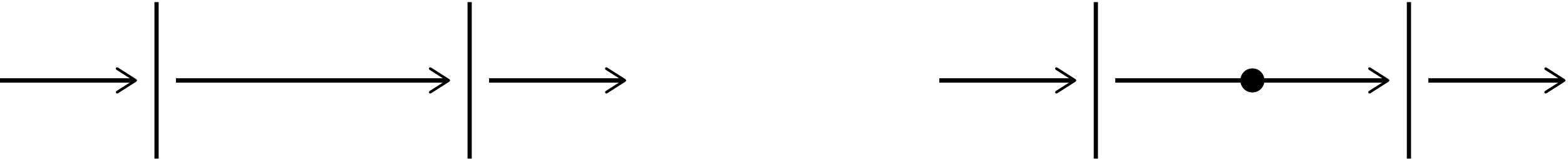}
      \put(57,-15){$D$}
      \put(-4,24){$a_{i\mathbf{p}(i)-1}$}
      \put(54,24){$a_{i\mathbf{p}(i)}$}
      \put(99,24){$a_{i\mathbf{p}(i)+1}$}
      \put(236,-15){$(D,\mathbf{p})$}
      \put(246,3){$p_{i}$}
      \put(184,24){$a_{i\mathbf{p}(i)-1}$}
      \put(233,24){$b_{i}^{\mathbf{p}}$}
      \put(255,24){$a_{i\mathbf{p}(i)}$}
      \put(286.5,24){$a_{i\mathbf{p}(i)+1}$}
    \end{overpic}
  \end{center}
  \vspace{1em}
  \caption{}
  \label{assign-b}
\end{figure}

In this setting, the homomorphism $\eta_{q}(D,\mathbf{p})$ associated with $(D,\mathbf{p})$ is described as follows. 
We put $\eta_{q}^{\mathbf{p}}=\eta_{q}(D,\mathbf{p})$ for short. 
The domain of $\eta_{q}^{\mathbf{p}}$ is the free group $\overline{A}$ on $\{a_{ij}\}\cup\{b_{i}^{\mathbf{p}}\}$. 
The homomorphism $\eta_{q}^{\mathbf{p}}$ from $\overline{A}$ into $A$ is given inductively by 
\begin{eqnarray*}
&&\eta_{1}^{\mathbf{p}}(a_{ij})=\alpha_{i},\ 
\eta_{1}^{\mathbf{p}}(b_{i}^{\mathbf{p}})=\alpha_{i}, \\
&&\eta_{q+1}^{\mathbf{p}}(a_{i\mathbf{p}(i)})=\alpha_{i},\ 
\eta_{q+1}^{\mathbf{p}}(a_{ij})=\eta_{q}^{\mathbf{p}}(({v}_{ij-1}^{\mathbf{p}})^{-1})\alpha_{i}\eta_{q}^{\mathbf{p}}(v_{ij-1}^{\mathbf{p}}) \hspace{1em} (j\neq \mathbf{p}(i)), \\
&&\mbox{and}\hspace{1em} 
\eta_{q+1}^{\mathbf{p}}(b_{i}^{\mathbf{p}})=\eta_{q}^{\mathbf{p}}((v_{i\mathbf{p}(i)-1}^{\mathbf{p}})^{-1})\alpha_{i}\eta_{q}^{\mathbf{p}}(v_{i\mathbf{p}(i)-1}^{\mathbf{p}}), 
\end{eqnarray*} 
where 
\[
v_{ij}^{\mathbf{p}}=
\begin{cases}
u_{i\mathbf{p}(i)}^{\varepsilon_{i\mathbf{p}(i)}}u_{i\mathbf{p}(i)+1}^{\varepsilon_{i\mathbf{p}(i)+1}}
\cdots u_{ij}^{\varepsilon_{ij}} & (\mathbf{p}(i)\leq j\leq m_{i}), \\[1em]
u_{i\mathbf{p}(i)}^{\varepsilon_{i\mathbf{p}(i)}}u_{i\mathbf{p}(i)+1}^{\varepsilon_{i\mathbf{p}(i)+1}}\cdots 
u_{im_{i}}^{\varepsilon_{im_i}}u_{i1}^{\varepsilon_{i1}}\cdots u_{ij}^{\varepsilon_{ij}} & (1\leq j\leq \mathbf{p}(i)-1),
\end{cases}
\] 
and $v_{i0}^{\mathbf{p}}=v_{im_i}^{\mathbf{p}}$. 
Furthermore, the $i$th preferred longitude $l_{i}^{\mathbf{p}}$ of $(D,\mathbf{p})$ is given by 
\[
l_{i}^{\mathbf{p}}=a_{i\mathbf{p}(i)}^{-w_{i}}v_{i\mathbf{p}(i)-1}^{\mathbf{p}}.
\]

We remark that the statement obtained from Lemma~\ref{lem-6k} by substituting 
 $\eta_{q}^{\mathbf{p}}$ for $\eta_{q}$ also holds.
Hereafter, even when applying the statement, we refer to Lemma~\ref{lem-6k}. 

We now define a word $\lambda_{i}^{\mathbf{p}}\in\overline{A}$ $(1\leq i\leq n)$ by 
\[
\lambda_{i}^{\mathbf{p}}=
\begin{cases}
u_{i1}^{\varepsilon_{i1}}u_{i2}^{\varepsilon_{i2}}\cdots u_{i\mathbf{p}(i)-1}^{\varepsilon_{i\mathbf{p}(i)-1}}
 & (\mathbf{p}(i)\neq 1), \\
1 & (\mathbf{p}(i)=1), 
\end{cases}
\] 
and a sequence of homomorphisms $\phi_{q}^{\mathbf{p}}:A\rightarrow A$ by 
\begin{eqnarray*}
&&\phi_{1}^{\mathbf{p}}(\alpha_{i})=\alpha_{i} \hspace{1em} \mbox{and}\\
&&\phi_{q}^{\mathbf{p}}(\alpha_{i})=\eta_{q-1}^{\mathbf{p}}(\lambda_{i}^{\mathbf{p}})\alpha_{i}\eta_{q-1}^{\mathbf{p}}(({\lambda_{i}^{\mathbf{p}}})^{-1}) 
\hspace{1em}
(q\geq2). 
\end{eqnarray*}
Notice that the homomorphism $\phi_{q}^{\mathbf{p}}$ sends each $\alpha_{i}$ to some conjugate element. 

\begin{lemma}\label{lem-phi}
For any $x,y\in A$, the following hold. 
\begin{enumerate}
\item $\phi_{q}^{\mathbf{p}}(x)\equiv\phi_{q+1}^{\mathbf{p}}(x) \pmod{A_{q}}$. 
\item If $x\equiv y\pmod{A_{q}}$, then $\phi_{q}^{\mathbf{p}}(x)\equiv\phi_{q}^{\mathbf{p}}(y) \pmod{A_{q}}$. 
\end{enumerate}
\end{lemma}

\begin{proof}
(1)~
Since $A_{1}=A$, this holds for $q=1$. 
For $q\geq2$, 
it is enough to show the case $x=\alpha_{i}$. 
By Lemma~\ref{lem-6k}(1), we have $\eta_{q-1}^{\mathbf{p}}(\lambda_{i}^{\mathbf{p}})\equiv \eta_{q}^{\mathbf{p}}(\lambda_{i}^{\mathbf{p}}) \pmod{A_{q-1}}$. 
Hence Lemma~\ref{lem-p290}(1) implies that 
\begin{eqnarray*}
\phi_{q}^{\mathbf{p}}(\alpha_{i})&=&\eta_{q-1}^{\mathbf{p}}(\lambda_{i}^{\mathbf{p}})\alpha_{i}\eta_{q-1}^{\mathbf{p}}(({\lambda_{i}^{\mathbf{p}}})^{-1}) \\
&\equiv& \eta_{q}^{\mathbf{p}}(\lambda_{i}^{\mathbf{p}})\alpha_{i}\eta_{q}^{\mathbf{p}}(({\lambda_{i}^{\mathbf{p}}})^{-1}) \pmod{A_{q}} \\
&=& \phi_{q+1}^{\mathbf{p}}(\alpha_{i}). 
\end{eqnarray*} 

(2)~
It is enough to show that $\phi_{q}^{\mathbf{p}}(A_{q})\subset A_{q}$. 
This is done by induction on $q$. 
For $q=1$ it is obvious. 
Assume that $q\geq1$. 
For $x\in A$ and $y\in A_{q}$, by (1) and Lemma~\ref{lem-p290}(1), it follows that 
\begin{eqnarray*}
\phi_{q+1}^{\mathbf{p}}([x,y])&=&[\phi_{q+1}^{\mathbf{p}}(x),\phi_{q+1}^{\mathbf{p}}(y)] \\
&\equiv& [\phi_{q+1}^{\mathbf{p}}(x),\phi_{q}^{\mathbf{p}}(y)] \pmod{A_{q+1}}. 
\end{eqnarray*}
By the induction hypothesis, we have $\phi_{q}^{\mathbf{p}}(y)\in A_{q}$ and hence $\phi_{q+1}^{\mathbf{p}}([x,y])\in A_{q+1}$. 
\end{proof}

A {\em semi-arc} of $D$ is a segment along $D$ which goes from a classical under-/over-crossing to the next one, where virtual crossings are ignored. 
Let $\mathcal{P}$ be the set of base point systems of $D$. 
Let $\mathcal{P}_{0}\subset\mathcal{P}$ be the set of all $(p_{1},\ldots,p_{n})\in\mathcal{P}$ such that each $p_{i}$ lies on a semi-arc which starts at a classical under-crossing. 
We denote by $\mathbf{p}_{*}=(p_{1}^{*},\ldots,p_{n}^{*})\in\mathcal{P}_{0}$ the base point system such that each $p_{i}^{*}$ lies on the arc $a_{i1}$. 
For the homomorphism $\eta_{q}^{\mathbf{p}_{*}}$ associated with $(D,\mathbf{p}_{*})$, partial longitudes $v_{ij}^{\mathbf{p}_{*}}$ and preferred longitudes $l_{i}^{\mathbf{p}_{*}}$ of $(D,\mathbf{p}_{*})$, 
we simply put $\eta_{q}=\eta_{q}^{\mathbf{p}_{*}}$, $v_{ij}=v_{ij}^{\mathbf{p}_{*}}$ and $l_{i}=l_{i}^{\mathbf{p}_{*}}$. 

Let $M_{q}^{\mathbf{p}}$ be the normal closure of $\{\phi_{q}^{\mathbf{p}}([\alpha_{i},\eta_{q}(l_{i})])\mid 1\leq i\leq n\}$ in $A$ 
and let $M_{q}=\prod_{\mathbf{p}\in\mathcal{P}_{0}} M_{q}^{\mathbf{p}}$. 
Notice that $M_{q}=\prod_{\mathbf{p}\in\mathcal{P}_{0}}\phi_{q}^{\mathbf{p}}(M_{q}^{\mathbf{p}_{*}})$. 

\begin{lemma}\label{lem-subset}
For any $\mathbf{p}\in\mathcal{P}$, $M_{q}^{\mathbf{p}}\subset A_{q}M_{q+1}^{\mathbf{p}}$. 
Hence $M_{q}\subset A_{q}M_{q+1}$. 
\end{lemma}

\begin{proof}
Lemma~\ref{lem-6k}(1) implies that 
$\eta_{q}(l_{i})\equiv\eta_{q+1}(l_{i})\pmod{A_{q}}$. 
Hence by Lemma~\ref{lem-phi}, we have 
\[
\phi_{q}^{\mathbf{p}}([\alpha_{i},\eta_{q}(l_{i})])
\equiv \phi_{q}^{\mathbf{p}}([\alpha_{i},\eta_{q+1}(l_{i})])
\equiv \phi_{q+1}^{\mathbf{p}}([\alpha_{i},\eta_{q+1}(l_{i})]) \pmod{A_{q}}. 
\] 
\end{proof}

\begin{lemma}\label{lem-diff-arc}
Let $\mathbf{p}_{0}\in\mathcal{P}_{0}$. 
For any $1\leq i\leq n$ and $1\leq j\leq m_{i}$, 
\[\eta_{q}^{\mathbf{p}_{0}}(a_{ij})\equiv\phi_{q}^{\mathbf{p}_{0}}(\eta_{q}(a_{ij})) \pmod{A_{q}M_{q}^{\mathbf{p}_{0}}}.\] 
\end{lemma}

\begin{proof}
This is proved by induction on $q$. 
It is obvious for $q=1$. 
Assume that $q\geq1$. 
The induction hypothesis together with Lemma~\ref{lem-subset} implies that 
\begin{eqnarray}\label{eq-diff-arc}
\eta_{q}^{\mathbf{p}_{0}}(a_{ij})\equiv\phi_{q}^{\mathbf{p}_{0}}(\eta_{q}(a_{ij})) \pmod{A_{q}M_{q+1}^{\mathbf{p}_{0}}}. 
\end{eqnarray}

First we consider the case $1\leq j\leq \mathbf{p}_{0}(i)-1$. 
For partial longitudes $v_{ij-1}^{\mathbf{p}_{0}}$ and $v_{ij-1}$ of $(D,\mathbf{p}_{0})$ and $(D,\mathbf{p}_{*})$, respectively, by definition we have 
\[
v_{ij-1}^{\mathbf{p}_{0}}=(\lambda_{i}^{\mathbf{p}_{0}})^{-1}v_{im_{i}}v_{ij-1}
=(\lambda_{i}^{\mathbf{p}_{0}})^{-1}a_{i1}^{w_{i}}l_{i}v_{ij-1}, 
\] 
where $v_{i0}^{\mathbf{p}_{0}}=(\lambda_{i}^{\mathbf{p}_{0}})^{-1}a_{i1}^{w_{i}}l_{i}$. 
Then it follows from congruence~(\ref{eq-diff-arc}) and Lemma~\ref{lem-p290}(2) that 
\begin{eqnarray*}
&&\eta_{q+1}^{\mathbf{p}_{0}}(a_{ij}) \\
&&= \eta_{q}^{\mathbf{p}_{0}}((v_{ij-1}^{\mathbf{p}_{0}})^{-1})\alpha_{i}\eta_{q}^{\mathbf{p}_{0}}(v_{ij-1}^{\mathbf{p}_{0}}) \\
&&= \eta_{q}^{\mathbf{p}_{0}}(v_{ij-1}^{-1}l_{i}^{-1}a_{i1}^{-w_{i}})
\left(\eta_{q}^{\mathbf{p}_{0}}(\lambda_{i}^{\mathbf{p}_{0}})\alpha_{i}\eta_{q}^{\mathbf{p}_{0}}((\lambda_{i}^{\mathbf{p}_{0}})^{-1})\right)
\eta_{q}^{\mathbf{p}_{0}}(a_{i1}^{w_{i}}l_{i}v_{ij-1}) \\
&&= \eta_{q}^{\mathbf{p}_{0}}(v_{ij-1}^{-1}l_{i}^{-1}a_{i1}^{-w_{i}})
\phi_{q+1}^{\mathbf{p}_{0}}(\alpha_{i})
\eta_{q}^{\mathbf{p}_{0}}(a_{i1}^{w_{i}}l_{i}v_{ij-1}) \\
&&\equiv \phi_{q}^{\mathbf{p}_{0}}(\eta_{q}(v_{ij-1}^{-1}l_{i}^{-1}a_{i1}^{-w_{i}}))
\phi_{q+1}^{\mathbf{p}_{0}}(\alpha_{i})
\phi_{q}^{\mathbf{p}_{0}}(\eta_{q}(a_{i1}^{w_{i}}l_{i}v_{ij-1})) 
\pmod{A_{q+1}M_{q+1}^{\mathbf{p}_{0}}} \\
&&= \phi_{q}^{\mathbf{p}_{0}}(\eta_{q}(v_{ij-1}^{-1}l_{i}^{-1})\alpha_{i}^{-w_{i}})
\phi_{q+1}^{\mathbf{p}_{0}}(\alpha_{i})
\phi_{q}^{\mathbf{p}_{0}}(\alpha_{i}^{w_{i}}\eta_{q}(l_{i}v_{ij-1})). 
\end{eqnarray*} 
On the other hand, by Lemmas~\ref{lem-p290}(1) and \ref{lem-6k}(1) we have 
\begin{eqnarray*}
&&\phi_{q+1}^{\mathbf{p}_{0}}(\eta_{q+1}(a_{ij})) \\
&&= \phi_{q+1}^{\mathbf{p}_{0}}(\eta_{q}(v_{ij-1}^{-1})\alpha_{i}\eta_{q}(v_{ij-1})) \\
&&\equiv \phi_{q+1}^{\mathbf{p}_{0}}(\eta_{q}(v_{ij-1}^{-1})\eta_{q}(l_{i}^{-1})
[\alpha_{i},\eta_{q+1}(l_{i})]\eta_{q}(l_{i})
\alpha_{i}\eta_{q}(v_{ij-1})) 
\pmod{M_{q+1}^{\mathbf{p}_{0}}} \\
&&\equiv \phi_{q+1}^{\mathbf{p}_{0}}(\eta_{q}(v_{ij-1}^{-1})\eta_{q}(l_{i}^{-1})
[\alpha_{i},\eta_{q}(l_{i})]
\eta_{q}(l_{i})\alpha_{i}\eta_{q}(v_{ij-1})) \pmod{A_{q+1}} \\
&&= \phi_{q+1}^{\mathbf{p}_{0}}(\eta_{q}(v_{ij-1}^{-1})\eta_{q}(l_{i}^{-1})\alpha_{i}\eta_{q}(l_{i})\eta_{q}(v_{ij-1})) \\
&&= \phi_{q+1}^{\mathbf{p}_{0}}(\eta_{q}(v_{ij-1}^{-1}l_{i}^{-1})\alpha_{i}^{-w_{i}})\phi_{q+1}^{\mathbf{p}_{0}}(\alpha_{i})\phi_{q+1}^{\mathbf{p}_{0}}(\alpha_{i}^{w_{i}}\eta_{q}(l_{i}v_{ij-1})). 
\end{eqnarray*} 
Therefore Lemmas~\ref{lem-p290}(1) and~\ref{lem-phi}(1) imply that 
\[
\eta_{q+1}^{\mathbf{p}_{0}}(a_{ij})
\equiv\phi_{q+1}^{\mathbf{p}_{0}}(\eta_{q+1}(a_{ij})) \pmod{A_{q+1}M_{q+1}^{\mathbf{p}_{0}}}.
\]

Next we consider the case $\mathbf{p}_{0}(i)\leq j\leq m_{i}$. 
For $j\neq {\mathbf{p}_{0}}(i)$, we have $v_{ij-1}=\lambda_{i}^{\mathbf{p}_{0}}v_{ij-1}^{\mathbf{p}_{0}}$. 
By congruence~(\ref{eq-diff-arc}) and Lemmas~\ref{lem-p290} and \ref{lem-phi}(1), it follows that 
\begin{eqnarray*}
\eta_{q+1}^{\mathbf{p}_{0}}(a_{ij})&=& \eta_{q}^{\mathbf{p}_{0}}((v_{ij-1}^{\mathbf{p}_{0}})^{-1})\alpha_{i} \eta_{q}^{\mathbf{p}_{0}}(v_{ij-1}^{\mathbf{p}_{0}}) \\
&=& \eta_{q}^{\mathbf{p}_{0}}((v_{ij-1}^{\mathbf{p}_{0}})^{-1}(\lambda_{i}^{\mathbf{p}_{0}})^{-1})
\left(\eta_{q}^{\mathbf{p}_{0}}(\lambda_{i}^{\mathbf{p}_{0}})\alpha_{i}\eta_{q}^{\mathbf{p}_{0}}((\lambda_{i}^{\mathbf{p}_{0}})^{-1})\right)
\eta_{q}^{\mathbf{p}_{0}}(\lambda_{i}^{\mathbf{p}_{0}}v_{ij-1}^{\mathbf{p}_{0}}) \\
&=& \eta_{q}^{\mathbf{p}_{0}}((v_{ij-1}^{\mathbf{p}_{0}})^{-1}(\lambda_{i}^{\mathbf{p}_{0}})^{-1})
\phi_{q+1}^{\mathbf{p}_{0}}(\alpha_{i})
\eta_{q}^{\mathbf{p}_{0}}(\lambda_{i}^{\mathbf{p}_{0}}v_{ij-1}^{\mathbf{p}_{0}}) \\
&=& \eta_{q}^{\mathbf{p}_{0}}(v_{ij-1}^{-1})
\phi_{q+1}^{\mathbf{p}_{0}}(\alpha_{i})
\eta_{q}^{\mathbf{p}_{0}}(v_{ij-1}) \\
&\equiv& \phi_{q}^{\mathbf{p}_{0}}(\eta_{q}(v_{ij-1}^{-1}))
\phi_{q+1}^{\mathbf{p}_{0}}(\alpha_{i})
\phi_{q}^{\mathbf{p}_{0}}(\eta_{q}(v_{ij-1})) \pmod{A_{q+1}M_{q+1}^{\mathbf{p}_{0}}} \\
&\equiv& \phi_{q+1}^{\mathbf{p}_{0}}(\eta_{q}(v_{ij-1}^{-1})\alpha_{i}\eta_{q}(v_{ij-1})) \pmod{A_{q+1}} \\
&=& \phi_{q+1}^{\mathbf{p}_{0}}(\eta_{q+1}(a_{ij})).
\end{eqnarray*} 
In the case $j= {\mathbf{p}_{0}}(i)$, by substituting 1 for $v_{ij-1}^ {\mathbf{p}_{0}}$ in the formula above, we have the conclusion. 
\end{proof}

\begin{proposition}\label{prop-diff-arc}
Let $\mathbf{p}_{0}\in\mathcal{P}_{0}$. 
For any $1\leq i\leq n$, 
\[
\eta_{q}^{\mathbf{p}_{0}}(l_{i}^{\mathbf{p}_{0}})
\equiv\phi_{q}^{\mathbf{p}_{0}}(\eta_{q}((\lambda_{i}^{\mathbf{p}_{0}})^{-1}l_{i}\lambda_{i}^{\mathbf{p}_{0}})) \pmod{A_{q}M_{q}^{\mathbf{p}_{0}}}.
\] 
\end{proposition}

\begin{proof}
Since $l_{i}^{\mathbf{p}_{0}}=a_{i\mathbf{p}_{0}(i)}^{-w_{i}}(\lambda_{i}^{\mathbf{p}_{0}})^{-1}a_{i1}^{w_{i}}l_{i}\lambda_{i}^{\mathbf{p}_{0}}$, it follows from Lemmas~\ref{lem-phi}(1) and \ref{lem-diff-arc} that 
\begin{eqnarray*}
\eta_{q}^{\mathbf{p}_{0}}(l_{i}^{\mathbf{p}_{0}})
&=&\eta_{q}^{\mathbf{p}_{0}}(a_{i\mathbf{p}_{0}(i)}^{-w_{i}}(\lambda_{i}^{\mathbf{p}_{0}})^{-1}a_{i1}^{w_{i}}l_{i}\lambda_{i}^{\mathbf{p}_{0}}) \\
&=& \alpha_{i}^{-w_{i}}\eta_{q}^{\mathbf{p}_{0}}((\lambda_{i}^{\mathbf{p}_{0}})^{-1}a_{i1}^{w_{i}}l_{i}\lambda_{i}^{\mathbf{p}_{0}}) \\
&\equiv& \alpha_{i}^{-w_{i}}\phi_{q}^{\mathbf{p}_{0}}(\eta_{q}((\lambda_{i}^{\mathbf{p}_{0}})^{-1}a_{i1}^{w_{i}}l_{i}\lambda_{i}^{\mathbf{p}_{0}})) \pmod{A_{q}M_{q}^{\mathbf{p}_{0}}} \\
&=& \alpha_{i}^{-w_{i}}\phi_{q}^{\mathbf{p}_{0}}(\eta_{q}((\lambda_{i}^{\mathbf{p}_{0}})^{-1}))
\phi_{q}^{\mathbf{p}_{0}}(\alpha_{i}^{w_{i}}) \phi_{q}^{\mathbf{p}_{0}}(\eta_{q}(l_{i}\lambda_{i}^{\mathbf{p}_{0}})) \\
&\equiv& \alpha_{i}^{-w_{i}}\phi_{q}^{\mathbf{p}_{0}}(\eta_{q}((\lambda_{i}^{\mathbf{p}_{0}})^{-1}))
\phi_{q+1}^{\mathbf{p}_{0}}(\alpha_{i}^{w_{i}}) \phi_{q}^{\mathbf{p}_{0}}(\eta_{q}(l_{i}\lambda_{i}^{\mathbf{p}_{0}})) \pmod{A_{q}} \\
&=& \alpha_{i}^{-w_{i}}\phi_{q}^{\mathbf{p}_{0}}(\eta_{q}((\lambda_{i}^{\mathbf{p}_{0}})^{-1})
\left(\eta_{q}^{\mathbf{p}_{0}}(\lambda_{i}^{\mathbf{p}_{0}})\alpha_{i}^{w_{i}}\eta_{q}^{\mathbf{p}_{0}}((\lambda_{i}^{\mathbf{p}_{0}})^{-1})\right) 
\phi_{q}^{\mathbf{p}_{0}}(\eta_{q}(l_{i}\lambda_{i}^{\mathbf{p}_{0}})) \\
&\equiv& \alpha_{i}^{-w_{i}}\phi_{q}^{\mathbf{p}_{0}}(\eta_{q}((\lambda_{i}^{\mathbf{p}_{0}})^{-1})
\phi_{q}^{\mathbf{p}_{0}}(\eta_{q}(\lambda_{i}^{\mathbf{p}_{0}})) \alpha_{i}^{w_{i}}\phi_{q}^{\mathbf{p}_{0}}(\eta_{q}((\lambda_{i}^{\mathbf{p}_{0}})^{-1})) \\
&&\times
\phi_{q}^{\mathbf{p}_{0}}(\eta_{q}(l_{i}\lambda_{i}^{\mathbf{p}_{0}})) \pmod{A_{q}M_{q}^{\mathbf{p}_{0}}} \\
&=& \phi_{q}^{\mathbf{p}_{0}}(\eta_{q}((\lambda_{i}^{\mathbf{p}_{0}})^{-1}l_{i}\lambda_{i}^{\mathbf{p}_{0}})). 
\end{eqnarray*} 
\end{proof}

\begin{lemma}\label{lem-same-ab}
Let $\mathbf{p}_{0}\in\mathcal{P}_{0}$. 
For any $1\leq i\leq n$, 
\[
\eta_{q}^{\mathbf{p}_{0}}(a_{i\mathbf{p}_{0}(i)})\equiv\eta_{q}^{\mathbf{p}_{0}}(b_{i}^{\mathbf{p}_{0}})
\pmod{A_{q}N_{q}^{\mathbf{p}_{0}}},
\] 
where $N_{q}^{\mathbf{p}_{0}}$ denotes the normal closure of $\{[\alpha_{i},\eta_{q}^{\mathbf{p}_{0}}(l_{i}^{\mathbf{p}_{0}})]\mid 1\leq i\leq n\}$ in $A$.
\end{lemma}

\begin{proof}
By Lemmas~\ref{lem-p290}(1) and~\ref{lem-6k}(1), it follows that
\begin{eqnarray*}
\eta_{q}^{\mathbf{p}_{0}}(a_{i\mathbf{p}_{0}(i)})
&\equiv& \eta_{q}^{\mathbf{p}_{0}}((l_{i}^{\mathbf{\mathbf{p}_{0}}})^{-1})[\alpha_{i},\eta_{q}^{\mathbf{p}_{0}}(l_{i}^{\mathbf{p}_{0}})]\eta_{q}^{\mathbf{p}_{0}}(l_{i}^{\mathbf{p}_{0}})\eta_{q}^{\mathbf{p}_{0}}(a_{i\mathbf{p}_{0}(i)}) \pmod{N_{q}^{\mathbf{p}_{0}}} \\
&=& \eta_{q}^{\mathbf{p}_{0}}((l_{i}^{\mathbf{p}_{0}})^{-1})\alpha_{i} \eta_{q}^{\mathbf{p}_{0}}(l_{i}^{\mathbf{p}_{0}}) \\
&\equiv& \eta_{q-1}^{\mathbf{p}_{0}}((l_{i}^{\mathbf{p}_{0}})^{-1})\alpha_{i} \eta_{q-1}^{\mathbf{p}_{0}}(l_{i}^{\mathbf{p}_{0}}) \pmod{A_{q}} \\
&=& \eta_{q-1}^{\mathbf{\mathbf{p}_{0}}}((v_{i\mathbf{p}_{0}(i)-1}^{\mathbf{p}_{0}})^{-1}a_{i\mathbf{p}_{0}(i)}^{w_{i}})\alpha_{i}\eta_{q-1}^{\mathbf{\mathbf{p}_{0}}}(a_{i\mathbf{p}_{0}(i)}^{-w_{i}}v_{i\mathbf{p}_{0}(i)-1}^{\mathbf{p}_{0}}) \\
&=& \eta_{q-1}^{\mathbf{\mathbf{p}_{0}}}((v_{i\mathbf{p}_{0}(i)-1}^{\mathbf{p}_{0}})^{-1})\alpha_{i}\eta_{q-1}^{\mathbf{\mathbf{p}_{0}}}(v_{i\mathbf{p}_{0}(i)-1}^{\mathbf{p}_{0}}) 
~=~ \eta_{q}^{\mathbf{\mathbf{p}_{0}}}(b_{i}^{\mathbf{p}_{0}}). 
\end{eqnarray*} 
\end{proof}

\begin{lemma}\label{lem-same-arc}
Let $\mathbf{p}\in\mathcal{P}$, and $\mathbf{p}_{0}\in\mathcal{P}_{0}$ with $\mathbf{p}_{0}(k)=\mathbf{p}(k)~(1\leq k\leq n)$. 
For any $1\leq i\leq n$, the following hold. 
\begin{enumerate}
\item For any $1\leq j\leq m_{i}$, 
$\eta_{q}^{\mathbf{p}}(a_{ij})\equiv\eta_{q}^{\mathbf{p}_{0}}(a_{ij})\pmod{A_{q}N_{q}^{\mathbf{p}_{0}}}$.  
\item 
$\eta_{q}^{\mathbf{p}}(b_{i}^{\mathbf{p}})\equiv\eta_{q}^{\mathbf{p}_{0}}(b_{i}^{\mathbf{p}_{0}})\pmod{A_{q}N_{q}^{\mathbf{p}_{0}}}$. 
\end{enumerate}
\end{lemma}

\begin{proof}
This is proved by induction on $q$. 
Since $A_{1}=A$, assertions (1) and (2) are obvious for $q=1$. 
Assume that $q\geq 1$.
 
(1)~
For $j=\mathbf{p}(i)$, we have $\eta_{q+1}^{\mathbf{p}}(a_{i\mathbf{p}(i)})=\alpha_{i}=\eta_{q+1}^{\mathbf{p}_{0}}(a_{i\mathbf{p}_{0}(i)})=\eta_{q+1}^{\mathbf{p}_{0}}(a_{i\mathbf{p}(i)})$ by definition. 
In a way similar to the proof of Lemma~\ref{lem-subset}, we have $N_{q}^{\mathbf{p}_{0}}\subset A_{q}N_{q+1}^{\mathbf{p}_{0}}$. 
Hence for $j\neq \mathbf{p}(i)$, the induction hypothesis implies that 
\[
\eta_{q}^{\mathbf{p}}(a_{ij})\equiv\eta_{q}^{\mathbf{p}_{0}}(a_{ij})\pmod{A_{q}N_{q+1}^{\mathbf{p}_{0}}} \hspace{1em} 
\mbox{and} \hspace{1em} 
\eta_{q}^{\mathbf{p}}(b_{i}^{\mathbf{p}})\equiv\eta_{q}^{\mathbf{p}_{0}}(b_{i}^{\mathbf{p}_{0}})\pmod{A_{q}N_{q+1}^{\mathbf{p}_{0}}}.
\] 
Furthermore, by Lemma~\ref{lem-same-ab} we have 
\[
\eta_{q}^{\mathbf{p}_{0}}(a_{i\mathbf{p}_{0}(i)})\equiv\eta_{q}^{\mathbf{p}_{0}}
(b_{i}^{\mathbf{p}_{0}})\pmod{A_{q}N_{q+1}^{\mathbf{p}_{0}}}.
\]
Here we note that $v_{ij-1}^{\mathbf{p}_{0}}$ does not contain the letters $b^{\mathbf{p}_{0}}_1,\ldots,b^{\mathbf{p}_{0}}_n$. 
Since $v_{ij-1}^{\mathbf{p}_{0}}$ is obtained from $v_{ij-1}^{\mathbf{p}}$ by replacing $b_{k}^{\mathbf{p}}$ with $a_{k\mathbf{p}_{0}(k)}$ $(1\leq k\leq n)$, 
we have 
\[
\eta_{q}^{\mathbf{p}}(v_{ij-1}^{\mathbf{p}})
\equiv \eta_{q}^{\mathbf{p}_{0}}(v_{ij-1}^{\mathbf{p}_{0}}) \pmod{A_{q}N_{q+1}^{\mathbf{p}_{0}}}. 
\]
Therefore by Lemma~\ref{lem-p290}(2) it follows that 
\begin{eqnarray*}
\eta_{q+1}^{\mathbf{p}}(a_{ij})
&=&\eta_{q}^{\mathbf{p}}((v_{ij-1}^{\mathbf{p}})^{-1})\alpha_{i}\eta_{q}^{\mathbf{p}}(v_{ij-1}^{\mathbf{p}}) \\
&\equiv& \eta_{q}^{\mathbf{p}_{0}}((v_{ij-1}^{\mathbf{p}_{0}})^{-1})\alpha_{i}\eta_{q}^{\mathbf{p}_{0}}(v_{ij-1}^{\mathbf{p}_{0}}) \pmod{A_{q+1}N_{q+1}^{\mathbf{p}_{0}}} \\
&=& \eta_{q+1}^{\mathbf{p}_{0}}(a_{ij}). 
\end{eqnarray*} 

(2)~
The proof is similar to that of (1).
\end{proof}

\begin{proposition}\label{prop-same-arc}
Let $\mathbf{p}\in\mathcal{P}$, and $\mathbf{p}_{0}\in\mathcal{P}_{0}$ with $\mathbf{p}_{0}(k)=\mathbf{p}(k)~(1\leq k\leq n)$.  
For any $1\leq i\leq n$, $\eta_{q}^{\mathbf{p}}(l_{i}^{\mathbf{p}})\equiv\eta_{q}^{\mathbf{p}_{0}}(l_{i}^{\mathbf{p}_{0}}) \pmod{A_{q}M_{q}^{\mathbf{p}_{0}}}$.
\end{proposition}

\begin{proof}
By Lemma~\ref{lem-same-arc}, we have 
$\eta_{q}^{\mathbf{p}}(v_{i\mathbf{p}(i)-1}^{\mathbf{p}})\equiv\eta_{q}^{\mathbf{p}_{0}}(v_{i\mathbf{p}_{0}(i)-1}^{\mathbf{p}_{0}}) \pmod{A_{q}N_{q}^{\mathbf{p}_{0}}}$. 
This implies that 
\begin{eqnarray*}
\eta_{q}^{\mathbf{p}}(l_{i}^{\mathbf{p}})=\eta_{q}^{\mathbf{p}}(a_{i\mathbf{p}(i)}^{-w_{i}}v_{i\mathbf{p}(i)-1}^{\mathbf{p}})
\equiv \eta_{q}^{\mathbf{p}_{0}}(a_{i\mathbf{p}_{0}(i)}^{-w_{i}}v_{i\mathbf{p}_{0}(i)-1}^{\mathbf{p}_{0}}) 
=\eta_{q}^{\mathbf{p}_{0}}(l_{i}^{\mathbf{p}_{0}}) 
\pmod{A_{q}N_{q}^{\mathbf{p}_{0}}}. 
\end{eqnarray*}
Hence, it is enough to show that $N_{q}^{\mathbf{p}_{0}}\subset A_{q}M_{q}^{\mathbf{p}_{0}}$, i.e. $[\alpha_{i},\eta_{q}^{\mathbf{p}_{0}}(l_{i}^{\mathbf{p}_{0}})]\in A_{q}M_{q}^{\mathbf{p}_{0}}$. 

By Lemma~\ref{lem-phi}(1), we have 
\[
\phi_{q}^{\mathbf{p}_{0}}(\alpha_{i})\equiv\phi_{q+1}^{\mathbf{p}_{0}}(\alpha_{i})\pmod{A_{q}}.
\]  
Since $\phi_{q+1}^{\mathbf{p}_{0}}(\alpha_{i})=\eta_{q}^{\mathbf{p}_{0}}(\lambda_{i}^{\mathbf{p}_{0}})\alpha_{i}\eta_{q}^{\mathbf{p}_{0}}((\lambda_{i}^{\mathbf{p}_{0}})^{-1})$, we have 
\[
\alpha_{i}\equiv\eta_{q}^{\mathbf{p}_{0}}((\lambda_{i}^{\mathbf{p}_{0}})^{-1})\phi_{q}^{\mathbf{p}_{0}}(\alpha_{i})\eta_{q}^{\mathbf{p}_{0}}(\lambda_{i}^{\mathbf{p}_{0}}) \pmod{A_{q}}.
\] 
Furthermore, Lemma~\ref{lem-diff-arc} and Proposition~\ref{prop-diff-arc} imply that
\begin{eqnarray*}
\eta_{q}^{\mathbf{p}_{0}}(l_{i}^{\mathbf{p}_{0}})
&\equiv& \phi_{q}^{\mathbf{p}_{0}}(\eta_{q}((\lambda_{i}^{\mathbf{p}_{0}})^{-1}l_{i}\lambda_{i}^{\mathbf{p}_{0}})) \pmod{A_{q}M_{q}^{\mathbf{p}_{0}}} \\
&\equiv& 
\eta_{q}^{\mathbf{p}_{0}}((\lambda_{i}^{\mathbf{p}_{0}})^{-1})\phi_{q}^{\mathbf{p}_{0}}(\eta_{q}(l_{i}))\eta_{q}^{\mathbf{p}_{0}}(\lambda_{i}^{\mathbf{p}_{0}}) \pmod{A_{q}M_{q}^{\mathbf{p}_{0}}}. 
\end{eqnarray*}  
Therefore it follows that 
\begin{eqnarray*}
[\alpha_{i},\eta_{q}^{\mathbf{p}_{0}}(l_{i}^{\mathbf{p}_{0}})]
&=& \alpha_{i}\eta_{q}^{\mathbf{p}_{0}}(l_{i}^{\mathbf{p}_{0}})\alpha_{i}^{-1}\eta_{q}^{\mathbf{p}_{0}}((l_{i}^{\mathbf{p}_{0}})^{-1}) \\
&\equiv& \left(\eta_{q}^{\mathbf{p}_{0}}((\lambda_{i}^{\mathbf{p}_{0}})^{-1})\phi_{q}^{\mathbf{p}_{0}}(\alpha_{i})\eta_{q}^{\mathbf{p}_{0}}(\lambda_{i}^{\mathbf{p}_{0}})
\right)
\left(\eta_{q}^{\mathbf{p}_{0}}((\lambda_{i}^{\mathbf{p}_{0}})^{-1})\phi_{q}^{\mathbf{p}_{0}}(\eta_{q}(l_{i}))\eta_{q}^{\mathbf{p}_{0}}(\lambda_{i}^{\mathbf{p}_{0}})
\right) \\
&& \times \left(\eta_{q}^{\mathbf{p}_{0}}((\lambda_{i}^{\mathbf{p}_{0}})^{-1})\phi_{q}^{\mathbf{p}_{0}}(\alpha_{i}^{-1})\eta_{q}^{\mathbf{p}_{0}}(\lambda_{i}^{\mathbf{p}_{0}})
\right) \\
&&\times \left(\eta_{q}^{\mathbf{p}_{0}}((\lambda_{i}^{\mathbf{p}_{0}})^{-1})\phi_{q}^{\mathbf{p}_{0}}(\eta_{q}(l_{i}^{-1}))\eta_{q}^{\mathbf{p}_{0}}(\lambda_{i}^{\mathbf{p}_{0}})
\right) \pmod{A_{q}M_{q}^{\mathbf{p}_{0}}} \\
&=& \eta_{q}^{\mathbf{p}_{0}}((\lambda_{i}^{\mathbf{p}_{0}})^{-1})\phi_{q}^{\mathbf{p}_{0}}(
[\alpha_{i},\eta_{q}(l_{i})])
\eta_{q}^{\mathbf{p}_{0}}(\lambda_{i}^{\mathbf{p}_{0}})\in A_{q}M_{q}^{\mathbf{p}_{0}}. 
\end{eqnarray*} 
\end{proof}

Combining Propositions~\ref{prop-diff-arc} and~\ref{prop-same-arc}, the following is obtained immediately. 

\begin{theorem}\label{prop-bpt-change}
Let $\mathbf{p}\in\mathcal{P}$, and $\mathbf{p}_{0}\in\mathcal{P}_{0}$ with $\mathbf{p}_{0}(k)=\mathbf{p}(k)~(1\leq k\leq n)$. 
For any $1\leq i\leq n$, 
$\eta_{q}^{\mathbf{p}}(l_{i}^{\mathbf{p}})\equiv\phi_{q}^{\mathbf{p}_{0}}(\eta_{q}((\lambda_{i}^{\mathbf{p}_{0}})^{-1}l_{i}\lambda_{i}^{\mathbf{p}_{0}})) \pmod{A_{q}M_{q}^{\mathbf{p}_{0}}}$. 
Hence  $\eta_{q}^{\mathbf{p}}(l_{i}^{\mathbf{p}})\equiv\phi_{q}^{\mathbf{p}_{0}}(\eta_{q}((\lambda_{i}^{\mathbf{p}_{0}})^{-1}l_{i}\lambda_{i}^{\mathbf{p}_{0}})) \pmod{A_{q}M_{q}}$.
\end{theorem}

\section{Milnor numbers and welded isotopy}
Let $D$ be an $n$-component virtual link diagram of a welded link $L$. 
As shown in Example~\ref{ex-Milnor-number}, the Milnor number $\mu_{(D,\mathbf{p})}(I)$ depends on the choice of a base point system $\mathbf{p}$ of $D$. 
Hence it is {\em not} an invariant of the welded link $L$. 
On the other hand, we show in this section that $\mu_{(D,\mathbf{p})}(I)$ modulo a certain indeterminacy is an invariant of~$L$ (Theorem~\ref{th-w-iso}). 

\begin{definition} 
For a sequence $i_{1}\ldots i_{r}$ of indices in $\{1,\ldots,n\}$,  
the {\em indeterminacy $\Delta_{(D,\mathbf{p})}(i_{1}\ldots i_{r})$} of $(D,\mathbf{p})$ is the greatest common divisor of all $\mu_{(D,\mathbf{p})}(j_{1}\ldots j_{s})$, where $j_{1}\ldots j_{s}$ $(2\leq s<r)$ is obtained from $i_{1}\ldots i_{r}$ by removing at least one index and permuting the remaining indices cyclicly. 
In particular, we set $\Delta_{(D,\mathbf{p})}(i_{1}i_{2})=0$. 
\end{definition}

\begin{theorem}\label{th-w-iso}
Let $D$ and $D'$ be virtual diagrams of a welded link. 
Let $\mathbf{p}$ and $\mathbf{p}'$ be base point systems of $D$ and $D'$, respectively. 
Then $\mu_{(D,\mathbf{p})}(I)\equiv\mu_{(D',\mathbf{p}')}(I)\pmod{\Delta_{(D,\mathbf{p})}(I)}$ and $\Delta_{(D,\mathbf{p})}(I)=\Delta_{(D',\mathbf{p}')}(I)$ for any sequence $I$.
\end{theorem}

This theorem guarantees the well-definedness of the following definition.

\begin{definition}
Let $L$ be an $n$-component welded link. 
For a sequence $I$ of indices in $\{1,\ldots,n\}$, 
the {\em Milnor $\omu$-invariant $\omu_{L}(I)$} of $L$ is the residue class of $\mu_{(D,\mathbf{p})}(I)$ modulo $\Delta_{(D,\mathbf{p})}(I)$ for any virtual diagram $D$ of $L$ and any base point system $\mathbf{p}$ of~$D$. 
\end{definition}

\begin{remark}
The Milnor $\omu$-invariant of welded links, defined above, coincides with the extension of Chrisman in~\cite{C} for any sequence. 
In particular, for classical links, the invariant coincides with the original one in~\cite{M57}. 
\end{remark}

In the remainder of this section, we fix $D$ and its arcs $a_{ij}$ $(1\leq i\leq n, 1\leq j\leq m_{i})$, and use the same notation as in Section~\ref{sec-bpt-change}. 
In this setting, the Milnor number $\mu_{(D,\mathbf{p})}(j_{1}\ldots j_{s}i)$ of $(D,\mathbf{p})$ is given by the coefficient of $X_{j_{1}}\cdots X_{j_{s}}$ in $E(\eta_{q}^{\mathbf{p}}(l_{i}^{\mathbf{p}}))$. 
For short, we put $\mu_{\mathbf{p}}(I)=\mu_{(D,\mathbf{p})}(I)$ and $\Delta_{\mathbf{p}}(I)=\Delta_{(D,\mathbf{p})}(I)$. 
In particular, we put $\mu(I)=\mu_{(D,\mathbf{p}_{*})}(I)$ and $\Delta(I)=\Delta_{(D,\mathbf{p}_{*})}(I)$.

For each $1\leq i\leq n$, we define a subset $\mathcal{D}_{i}$ of $\mathbb{Z}\langle\langle X_{1},\ldots,X_{n}\rangle\rangle$ to be
\[
\left\{\sum\nu(j_{1}\ldots j_{s})X_{j_{1}}\cdots X_{j_{s}}~\left|
\begin{array}{ll}
\nu(j_{1}\ldots j_{s})\equiv0\pmod{\Delta(j_{1}\ldots j_{s}i)} &(s<q), \\
\nu(j_{1}\ldots j_{s})\in\mathbb{Z} &(s\geq q).
\end{array}
\right.
\right\}.
\] 

Although the following three results, Sublemmas~\ref{lem-ideal}, \ref{lem-conj} and Lemma~\ref{lem-Delta}, 
are essentially shown in \cite{M57}, we give the proofs for the readers' convenience. 
We use Sublemmas~\ref{lem-ideal} and~\ref{lem-conj} to prove Lemma~\ref{lem-Delta}. 

\begin{sublemma}[{cf. \cite[(14) and (16)--(19) on pages 292 and 293]{M57}}]
\label{lem-ideal}
For any $1\leq i\leq n$, the following hold. 
\begin{enumerate}
\item $\mathcal{D}_{i}$ is a two-sided ideal of $\mathbb{Z}\langle\langle X_{1},\ldots,X_{n}\rangle\rangle$. 
\item Let $f_{k}=E(\eta_{q}(l_{k}))-1$ $(1\leq k\leq n)$. 
Then $X_{j}f_{i},f_{i}X_{j}\in\mathcal{D}_{i}$ and $X_{j}f_{j},f_{j}X_{j}\in\mathcal{D}_{i}$ for any $1\leq j\leq n$. 
\item Let $k_{1}\ldots k_{s+t}$ be a sequence obtained from a sequence $j_{1}\ldots j_{s}$ by inserting $t~(\geq1)$ indices in $\{1,\ldots,n\}$. 
Then $\mu(j_{1}\ldots j_{s}i)X_{k_{1}}\cdots X_{k_{s+t}}\in\mathcal{D}_{i}$. 
\item $E([\alpha_{j},\eta_{q}(l_{j})])-1\in\mathcal{D}_{i}$ for any $1\leq j\leq n$. 
\end{enumerate}
\end{sublemma}

\begin{proof}
(1)~
Let $\nu(j_{1}\ldots j_{s}) X_{j_{1}}\cdots X_{j_{s}}\in\mathcal{D}_{i}$ and $X_{k}\in\{X_{1},\ldots,X_{n}\}$. 
For $s+1\geq q$, we have 
$\nu(j_{1}\ldots j_{s}) X_{k}X_{j_{1}}\cdots X_{j_{s}}\in\mathcal{D}_{i}$ 
by definition. 
For $s+1<q$, since 
$\Delta(j_{1}\ldots j_{s}i)$ is divisible by 
$\Delta(kj_{1}\ldots j_{s}i)$, 
$\nu(j_{1}\ldots j_{s})$ is also divisible by $\Delta(kj_{1}\ldots j_{s}i)$. 
Hence $\nu(j_{1}\ldots j_{s}) X_{k}X_{j_{1}}\cdots X_{j_{s}}\in\mathcal{D}_{i}$. 

Similarly, we have $\nu(j_{1}\ldots j_{s}) X_{j_{1}}\cdots X_{j_{s}}X_{k}\in\mathcal{D}_{i}$. 

(2)~
By definition, it follows that 
$f_{k}=E(\eta_{q}(l_{k}))-1$ has the form 
\[\sum\mu(j_{1}\ldots j_{s}k)X_{j_{1}}\cdots X_{j_{s}}.
\]
Since both $\mu(j_{1}\ldots j_{s}i)$ and $\mu(j_{1}\ldots j_{s}j)$ are divisible by $\Delta(jj_{1}\ldots j_{s}i)$, 
we have 
$\mu(j_{1}\ldots j_{s}i)X_{j}X_{j_{1}}\cdots X_{j_{s}}\in\mathcal{D}_{i}$ and 
$\mu(j_{1}\ldots j_{s}j)X_{j}X_{j_{1}}\cdots X_{j_{s}}\in\mathcal{D}_{i}$. 
Hence $X_{j}f_{i}\in\mathcal{D}_{i}$ and $X_{j}f_{j}\in\mathcal{D}_{i}$.

Similarly, we have $f_{i}X_{j}\in\mathcal{D}_{i}$ and $f_{j}X_{j}\in\mathcal{D}_{i}$. 

(3)~
Since $\mu(j_{1}\ldots j_{s}i)$ is divisible by $\Delta(k_{1}\ldots k_{s+t}i)$, 
we have the conclusion. 

(4)~
By a direct computation, 
it follows that 
\begin{eqnarray*}
E([\alpha_{j},\eta_{q}(l_{j})])-1
&=& \left(
E(\alpha_{j}\eta_{q}(l_{j}))-E(\eta_{q}(l_{j})\alpha_{j})
\right)
E(\alpha_{j}^{-1}\eta_{q}((l_{j})^{-1})) \\
&=& \left(
(1+X_{j})(1+f_{j})-(1+f_{j})(1+X_{j})
\right)
E(\alpha_{j}^{-1}\eta_{q}((l_{j})^{-1})) \\
&=& (X_{j}f_{j}-f_{j}X_{j})E(\alpha_{j}^{-1}\eta_{q}((l_{j})^{-1})). 
\end{eqnarray*}
Hence, by (1) and (2), we have 
$E([\alpha_{j},\eta_{q}(l_{j})])-1\in\mathcal{D}_{i}$.
\end{proof}

\begin{sublemma}[{cf. \cite[(12) on page 292]{M57}}]
\label{lem-conj}
Let $x\in A$ and $\mathbf{p}\in\mathcal{P}$.
If $E(x)-1\in\mathcal{D}_{i}$, then $E(\phi_{q}^{\mathbf{p}}(x))-1\in\mathcal{D}_{i}$. 
\end{sublemma}

\begin{proof}
Put $E(x)-1=\sum\nu(j_{1}\ldots j_{s})X_{j_{1}}\cdots X_{j_{s}}\in\mathcal{D}_{i}$, $E(\eta_{q-1}^{\mathbf{p}}(\lambda_{j}^{\mathbf{p}}))=1+h_{j}$ and $E(\eta_{q-1}^{\mathbf{p}}((\lambda_{j}^{\mathbf{p}})^{-1}))=1+\overline{h_{j}}$. 
Then we have 
\begin{eqnarray*}
E(\phi_{q}^{\mathbf{p}}(\alpha_{j}))-1
=X_{j}+X_{j}\overline{h_{j}}+h_{j}X_{j}+h_{j}X_{j}\overline{h_{j}}. 
\end{eqnarray*}
This shows that 
$E(\phi_{q}^{\mathbf{p}}(x))-1$ is obtained from $E(x)-1$ by replacing $X_{j}$ with $X_j+X_{j}\overline{h_{j}}+h_{j}X_{j}+h_{j}X_{j}\overline{h_{j}}$. 
Therefore 
$E(\phi_{q}^{\mathbf{p}}(x))-1$
can be written in the form 
\begin{eqnarray*}
\sum\nu(j_{1}\ldots j_{s})\left(X_{j_{1}}\cdots X_{j_{s}}
+\sum\kappa(k_{1}\ldots k_{s+t})X_{k_{1}}\cdots X_{k_{s+t}}
\right), 
\end{eqnarray*}
where the sequences $k_{1}\ldots k_{s+t}$ are obtained from the sequence $j_{1}\ldots j_{s}$ by inserting $t~(\geq1)$ indices in $\{1,\ldots,n\}$. 
Since $\Delta(j_{1}\ldots j_{s}i)$ is divisible by $\Delta(k_{1}\ldots k_{s+t}i)$, 
$\nu(j_{1}\ldots j_{s})$ is also divisible by $\Delta(k_{1}\ldots k_{s+t}i)$. 
Hence $\nu(j_{1}\ldots j_{s})X_{k_{1}}\cdots X_{k_{s+t}}\in\mathcal{D}_{i}$. 
This implies that $E(\phi_{q}^{\mathbf{p}}(x))-1\in\mathcal{D}_{i}$. 
\end{proof}

\begin{lemma}[{cf.~\cite[(12)--(15) on page 292]{M57}}]
\label{lem-Delta}
Let $x,y\in A$ and $\mathbf{p}\in\mathcal{P}$. 
For any $1\leq i\leq n$, the following hold. 
\begin{enumerate}
\item $E(x^{-1}\eta_{q}(l_{i})x)-E(\eta_{q}(l_{i}))\in\mathcal{D}_{i}$. 
\item $E(\phi_{q}^{\mathbf{p}}(\eta_{q}(l_{i})))-E(\eta_{q}(l_{i}))\in\mathcal{D}_{i}$. 
\item If $x\equiv y\pmod{A_{q}M_{q}}$, then $E(x)-E(y)\in\mathcal{D}_{i}$. 
\end{enumerate}
\end{lemma}

\begin{proof}
(1)~
Put $E(\eta_{q}(l_{i}))=1+f_{i}$ and $E(x)=1+h$. 
Then by Sublemma~\ref{lem-ideal}(1),~(2), we have  
\begin{eqnarray*}
E(x^{-1}\eta_{q}(l_{i})x)-E(\eta_{q}(l_{i}))
&=& E(x^{-1})\left(E(\eta_{q}(l_{i})x)-E(x\eta_{q}(l_{i}))\right) \\
&=& E(x^{-1})\left((1+f_{i})(1+h)-(1+h)(1+f_{i})\right) \\
&=& E(x^{-1})(f_{i}h-hf_{i})\in\mathcal{D}_{i}. 
\end{eqnarray*}

(2)~
Since $E(\eta_{q}(l_{i}))-1=\sum\mu(j_{1}\ldots j_{s}i)X_{j_{1}}\cdots X_{j_{s}}$, 
it follows from the proof of Sublemma~\ref{lem-conj} that 
\begin{eqnarray*}
E(\phi_{q}^{\mathbf{p}}(\eta_{q}(l_{i})))-E(\eta_{q}(l_{i})) 
=\sum\mu(j_{1}\ldots j_{s}i)\left(\sum \kappa(k_{1}\ldots k_{s+t})X_{k_{1}}\cdots X_{k_{s+t}}\right), 
\end{eqnarray*}
where the sequences $k_{1}\ldots k_{s+t}$ are obtained from the sequence $j_{1}\ldots j_{s}$ by inserting $t~(\geq1)$ indices in $\{1,\ldots,n\}$. 
By Sublemma~\ref{lem-ideal}(3), 
we have
$\mu(j_{1}\ldots j_{s}i)X_{k_{1}}\cdots X_{k_{s+t}}\in\mathcal{D}_{i}$, 
and hence $E(\phi_{q}^{\mathbf{p}}(\eta_{q}(l_{i})))-E(\eta_{q}(l_{i}))\in\mathcal{D}_{i}$. 

(3)~
It is enough to show that if $x\in A_{q}M_{q}$, then we have $E(x)-1\in\mathcal{D}_{i}$. 
For $x\in A_{q}$, it is true by Lemma~\ref{lem-Magnus}. 
For $x\in M_{q}$, we only need to consider the case $x=\phi_{q}^{\mathbf{p}_{0}}([\alpha_{j},\eta_{q}(l_{j})])$ $(\mathbf{p}_{0}\in\mathcal{P}_{0}, 1\leq j\leq n)$. 
By Sublemmas~\ref{lem-ideal}(4) and \ref{lem-conj}, 
we have $E(\phi_{q}^{\mathbf{p}_{0}}([\alpha_{j},\eta_{q}(l_{j})]))-1\in\mathcal{D}_{i}$. 
\end{proof}

\begin{proposition}\label{prop-Delta}
For any $\mathbf{p}\in\mathcal{P}$, the following hold. 
\samepage
\begin{enumerate}
\item $\mu_{\mathbf{p}}(I)\equiv\mu(I) \pmod{\Delta(I)}$ for any sequence $I$. 
\item $\Delta_{\mathbf{p}}(I)=\Delta(I)$ for any sequence $I$.
\end{enumerate}
\end{proposition}

\begin{proof}
(1)~
For any $1\leq i\leq n$, 
it is enough to show that 
\[
E(\eta_{q}^{\mathbf{p}}(l_{i}^{\mathbf{p}}))-E(\eta_{q}(l_{i}))\equiv0\pmod{\mathcal{D}_{i}}. 
\]  
Let $\mathbf{p}_{0}\in\mathcal{P}_{0}$ with $\mathbf{p}_{0}(i)=\mathbf{p}(i)$. 
By Theorem~\ref{prop-bpt-change}, we have 
\[
\eta_{q}^{\mathbf{p}}(l_{i}^{\mathbf{p}})
\equiv\phi_{q}^{\mathbf{p}_{0}}(\eta_{q}((\lambda_{i}^{\mathbf{p}_{0}})^{-1}l_{i}\lambda_{i}^{\mathbf{p}_{0}})) \pmod{A_{q}M_{q}}. 
\]
Put $x=\phi_{q}^{\mathbf{p}_{0}}(\eta_{q}(\lambda_{i}^{\mathbf{p}_{0}}))\in A$. 
Then by Lemma~\ref{lem-Delta} it follows that 
\begin{eqnarray*}
E(\eta_{q}^{\mathbf{p}}(l_{i}^{\mathbf{p}}))-E(\eta_{q}(l_{i})) 
&\equiv& E(x^{-1}\phi_{q}^{\mathbf{p}_{0}}(\eta_{q}(l_{i}))x)-E(\eta_{q}(l_{i})) \pmod{\mathcal{D}_{i}} \\
&\equiv& E(x^{-1}\phi_{q}^{\mathbf{p}_{0}}(\eta_{q}(l_{i}))x)-E(x^{-1}\eta_{q}(l_{i})x) \pmod{\mathcal{D}_{i}} \\
&=& E(x^{-1})
\left(E(\phi_{q}^{\mathbf{p}_{0}}(\eta_{q}(l_{i})))-E(\eta_{q}(l_{i}))
\right)E(x) \\
&\equiv& 0\pmod{\mathcal{D}_{i}}. 
\end{eqnarray*} 
Since we may assume that $q$ is sufficiently large, 
\[\mu_{\mathbf{p}}(j_1\ldots j_si)-\mu(j_1\ldots j_si)\equiv 0 \pmod{\Delta(j_1\ldots j_si)}
\]
for any sequence $j_{1}\ldots j_{s}i$.

(2)~This is proved by induction on the length $k$ of $I$. 
For $k=2$, we have $\Delta_{\mathbf{p}}(I)=\Delta(I)=0$ by definition. 
Assume that $k\geq2$. 
Let $\mathcal{J}_{1}(I)$ (resp. $\mathcal{J}_{\geq1}(I)$) be the set of all sequences obtained from $I$ by removing exactly one index (resp. at least one index) and permuting the remaining indices cyclicly. 
For any $J\in\mathcal{J}_{1}(I)$, we have $\Delta_{\mathbf{p}}(J)=\Delta(J)$ by the induction hypothesis. 
Then it follows that 
\begin{eqnarray*}
\Delta_{\mathbf{p}}(I)
&=& \gcd{\left\{\mu_{\mathbf{p}}(J)\mid J\in\mathcal{J}_{\geq1}(I)\right\}} \\
&=&\gcd{\left(
\bigcup_{J\in\mathcal{J}_{1}(I)}
\left(
\{\mu_{\mathbf{p}}(J)\}\cup\{\mu_{\mathbf{p}}(J')\mid J'\in\mathcal{J}_{\geq1}(J)\}
\right)
\right)} \\
&=&\gcd{\left(
\bigcup_{J\in\mathcal{J}_{1}(I)}
\left(
\{\mu_{\mathbf{p}}(J)\}\cup\{\Delta_{\mathbf{p}}(J)\}
\right)
\right)} \\
&=&\gcd{\left(
\bigcup_{J\in\mathcal{J}_{1}(I)}
\left(
\{\mu_{\mathbf{p}}(J)\}\cup\{\Delta(J)\}
\right)
\right)}.
\end{eqnarray*}
By~(1) we have $\mu_{\mathbf{p}}(J)\equiv\mu(J)\pmod{\Delta(J)}$. 
This implies that $\Delta_{\mathbf{p}}(I)=\Delta(I)$. 
\end{proof}

\begin{proof}[{Proof of Theorem~\ref{th-w-iso}}]
Since $(D,\mathbf{p})$ and $(D',\mathbf{p}')$ are related by $\overline{\mbox{w}}$-isotopies and base-change moves in Figure~\ref{bpt-change}, 
this follows from Theorem~\ref{th-w-bar} and Proposition~\ref{prop-Delta}. 
\end{proof}

\section{Self-crossing virtualization}
A {\em self-crossing virtualization} is a local move on virtual link diagrams as shown in Figure~\ref{SV}, which replaces a classical crossing involves two strands of a single component with a virtual one. 
In this section, we show the following theorem as a generalization of \cite[Theorem 8]{M57}. 

\begin{figure}[htb]
  \begin{center}
    \begin{overpic}[width=7cm]{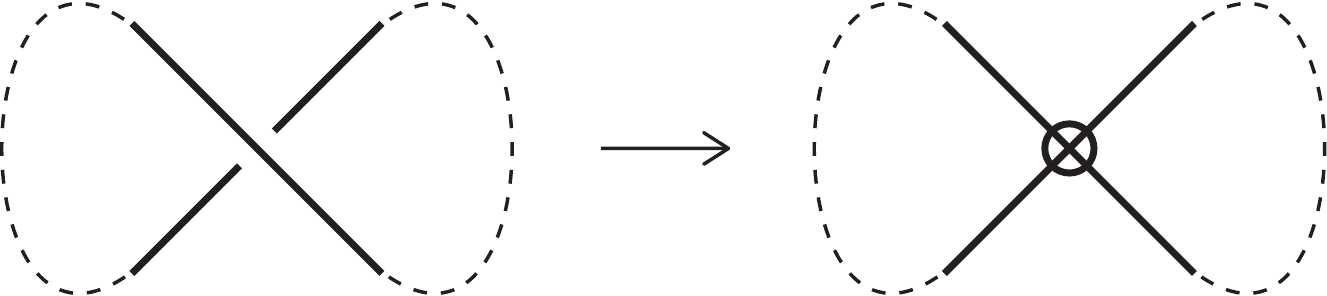}
    \end{overpic}
  \end{center}
  \caption{Self-crossing virtualization}
  \label{SV}
\end{figure}

\begin{theorem}\label{th-sv-link}
Let $L$ and $L'$ be welded links, and 
let $D$ and $D'$ be virtual link diagrams of $L$ and $L'$, respectively. 
If $D$ and $D'$ are related by a finite sequence  of self-crossing virtualizations and 
welded isotopies, then $\omu_{L}(I)=\omu_{L'}(I)$ 
for any non-repeated sequence $I$.
\end{theorem}

\begin{remark}
In~\cite{AM}, Audoux and Meilhan proved that two virtual link diagrams are related by a finite sequence of self-crossing virtualizations and 
welded isotopies if and only if they have equivalent {\em reduced peripheral systems}. 
This result together with Theorem~\ref{th-sv-link} implies that for welded links, the reduced peripheral system determines Milnor $\omu$-invariants for non-repeated sequences. 
\end{remark}

Let $(D,\mathbf{p})$ be an $n$-component virtual link diagram with a base point system, and let $a_{ij}$ $(1\leq i\leq n, 1\leq j\leq m_i+1)$ be the arcs of $(D,\mathbf{p})$ as given in Section~\ref{sec-prelim}. 
Recall that $A=\langle\alpha_{1},\ldots,\alpha_{n}\rangle$ denotes the free group of rank $n$, and $\overline{A}$ denotes the free group on $\{a_{ij}\}$. 
For $1\leq k\leq n$, let $A^{(k)}=\langle\alpha_{1},\ldots,\alpha_{k-1},\alpha_{k+1},\ldots,\alpha_{n}\rangle$ be the free group of rank $n-1$. 
We define a homomorphism 
$\rho_{k}:A\rightarrow A^{(k)}$ by 
\[
\rho_{k}(\alpha_{i})=
\begin{cases}
\alpha_{i} & (i\neq k), \\
1 & (i=k), 
\end{cases}
\]
and denote by $\eta_{q}^{(k)}=\eta_{q}^{(k)}(D,\mathbf{p})$ the composition $\rho_{k}\circ\eta_{q}:\overline{A}
\rightarrow A^{(k)}.$

\begin{lemma}\label{lem-eta}
The following hold. 
\begin{enumerate}
\item  
$\eta_{q}^{(k)}(a_{kj})=1$ for any $j$. 

\item For any $x\in\overline{A}$, $E(\eta_{q}^{(k)}(x))=E(\eta_{q}(x))|_{X_{k}=0}$, 
where  
$E(\cdot)|_{X_{k}=0}$ denotes the formal power series obtained from $E(\cdot)$ by 
substituting $0$ for $X_{k}$. 
\end{enumerate}
\end{lemma}

\begin{proof}
(1)~
Since $\eta_q(a_{kj})$ is a conjugate of $\alpha_k$, 
$\eta_q^{(k)}(a_{kj})$ is a conjugate of $\rho_k(\alpha_k)=1$.

(2)~
It is enough to show the case $x=a_{ij}$. 
This is proved by induction on $q$. 
The assertion certainly holds for $q=1$ or $j=1$. 
Assume that $q\geq1$ and $j\geq2$.
Then we have 
$E(\eta_{q}^{(k)}(v_{ij-1}))=E(\eta_{q}(v_{ij-1}))|_{X_{k}=0}$ 
by the induction hypothesis. 
Hence it follows that 
\begin{eqnarray*}
E(\eta_{q+1}^{(k)}(a_{ij}))&=&E(\rho_{k}(\eta_{q}(v_{ij-1}^{-1})\alpha_{i}\eta_{q}(v_{ij-1}))) \\ 
&=& E(\eta_{q}^{(k)}(v_{ij-1}^{-1}))(E(\alpha_{i})|_{X_{k}=0})E(\eta_{q}^{(k)}(v_{ij-1})) \\
&=& E(\eta_{q}(v_{ij-1}^{-1})\alpha_{i}\eta_{q}(v_{ij-1}))|_{X_{k}=0} \\
&=& E(\eta_{q+1}(a_{ij}))|_{X_{k}=0}. 
\end{eqnarray*}
\end{proof}

Let $R$ be the normal closure of $\{[\alpha_{i},g^{-1}\alpha_{i}g]\mid g\in A, 1\leq i\leq n\}$ in $A$, 
and let $R^{(k)}$ be the normal closure of $\{[\alpha_{i},g^{-1}\alpha_{i}g]\mid g\in A^{(k)}, 1\leq i\neq k\leq n\}$ in $A^{(k)}$. 
Note that $[g^{-1}\alpha_{i}g,h^{-1}\alpha_{i}h]\in R$ for any $g,h\in A$. 
In particular, $\eta_{q}([a_{is}^{\e},a_{it}^{\de}])\in R$ for any $s,t$ and any $\e,\de\in\{\pm1\}$. 
Let $A_{q}^{(k)}$ be the $q$th term of the lower central series of $A^{(k)}$. 
Then we have the following.

\begin{lemma}\label{lem-rho}
Let $x,y\in A$. If $x\equiv y\pmod{A_{q}R}$, then 
\[\rho_{k}(x)\equiv\rho_{k}(y) \pmod{A_{q}^{(k)}R^{(k)}}.
\]
\end{lemma} 

\begin{proof}
Since $\rho_{k}([\alpha_{i},g^{-1}\alpha_{i}g])\in R^{(k)}$ for any $g\in A$ and $1\leq i\leq n$, we have $\rho_{k}(R)\subset R^{(k)}$. 
Therefore it is enough to show that $\rho_{k}(A_{q})\subset A_{q}^{(k)}$. 

This is proved by induction on $q$. 
For $q=1$ the assertion is obvious. 
Assume that $q\geq1$. 
For $[x,y]\in A_{q+1}$ with $x\in A$ and $y\in A_{q}$, 
by the induction hypothesis, we have 
$\rho_{k}([x,y])=[\rho_{k}(x),\rho_{k}(y)]\in A_{q+1}^{(k)}$. 
\end{proof}

\begin{proposition}\label{prop-sv-eta}
Let $(D,\mathbf{p})$ and $(D',\mathbf{p}')$ be $n$-component virtual link diagrams with base point systems. 
For an integer $k\in\{1,\ldots,n\}$, let $l_{k}$ and $l'_{k}$ be the $k$th preferred longitudes of $(D,\mathbf{p})$ and $(D',\mathbf{p}')$, respectively. 
If $(D,\mathbf{p})$ and $(D',\mathbf{p}')$ are related by a self-crossing virtualization, then 
$\eta_{q}^{(k)}(D,\mathbf{p})(l_{k})\equiv\eta_{q}^{(k)}(D',\mathbf{p}')(l'_{k}) \pmod{A_{q}^{(k)}R^{(k)}}$. 
\end{proposition} 

\begin{proof}
Assume that $(D,\mathbf{p})$ and $(D',\mathbf{p}')$ are identical except in a disk~$\delta$, where they differ as shown in Figure~\ref{pf-SV}. 
Furthermore, without loss of generality, we may assume that the self-crossing virtualization in the figure is applied to the $1$st component. 

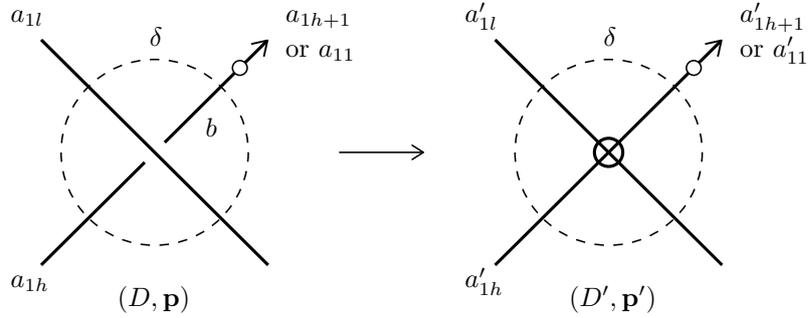
\begin{figure}[htb]
  \begin{center}
    \vspace{1em}
    \begin{overpic}[width=9cm]{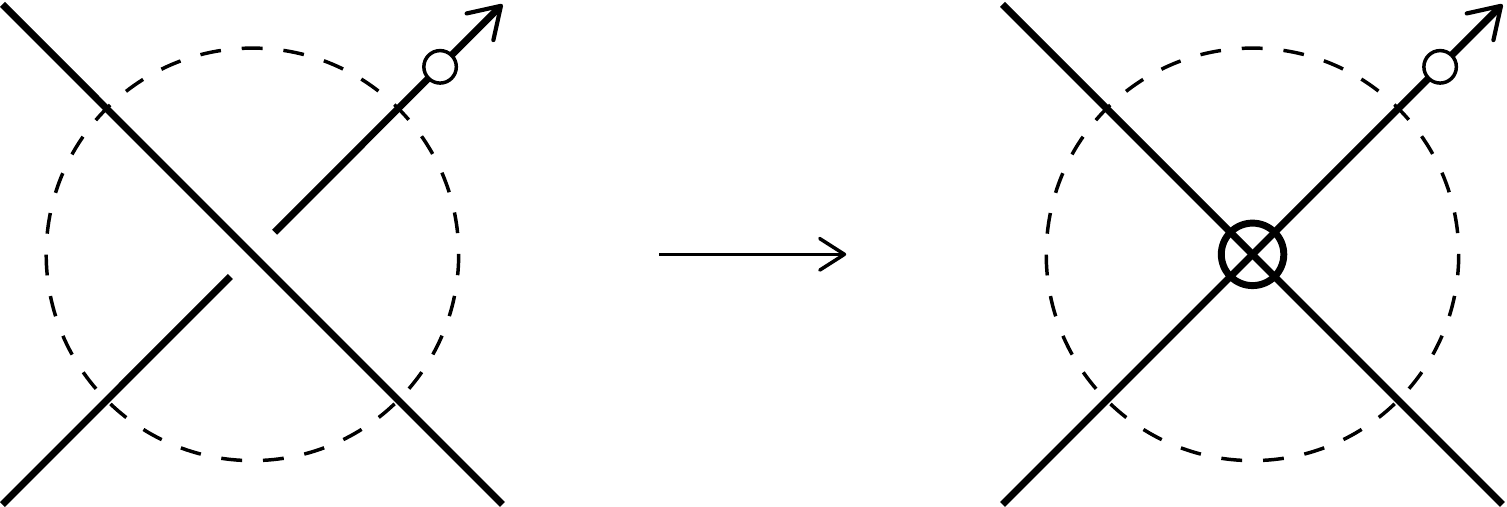}
      \put(29,-15){$(D,\mathbf{p})$}
      \put(41,83){$\delta$}
      \put(198,-15){$(D',\mathbf{p}')$}
      \put(211,83){$\delta$}
      \put(-11,92){$a_{1l}$}
      \put(-11,-8){$a_{1h}$}
      \put(92,92){$a_{1h+1}$}
      \put(92,79){or $a_{11}$}
      \put(62,49){$b$}
      \put(159,92){$a'_{1l}$}
      \put(159,-8){$a'_{1h}$}
      \put(262,92){$a'_{1h+1}$}
      \put(262,79){or $a'_{11}$}
    \end{overpic}
  \end{center}
  \vspace{1em}
  \caption{A self-crossing virtualization which relates $(D,\mathbf{p})$ to $(D',\mathbf{p}')$}
  \label{pf-SV}
\end{figure}

Let $a'_{ij}$ $(1\leq i\leq n, 1\leq j\leq m_i+1)$ be the arcs of $(D',\mathbf{p}')$ as given in Section~\ref{sec-prelim}. 
Let $a_{ij}$ and $b$ be the arcs of $(D,\mathbf{p})$ such that each $a_{ij}$ corresponds to the arc $a'_{ij}$ of $(D',\mathbf{p}')$, and $b$ intersects the disk $\delta$ as shown in Figure~\ref{pf-SV}. 
Note that the domain of $\eta_{q}^{(k)}(D,\mathbf{p})$ is the free group $\overline{A}$ on $\{a_{ij}\}\cup \{b\}$. 
Let $\e\in\{\pm1\}$ be the sign of the classical crossing among $a_{1h},a_{1l}$ and $b$ in Figure~\ref{pf-SV}. 
Since the self-crossing virtualization is applied to the $1$st component, 
$l_{1}$ and $l'_{1}$ can be written in the forms 
\[
l_{1}=a_{11}^{-w_{1}}xa_{1l}^{\e}y\hspace{1em} \mbox{and}\hspace{1em}  l'_{1}=(a'_{11})^{-w_{1}-\e}x'y'
\]
for certain words $x,y\in \overline{A}$ and $x',y'\in\overline{A'}$ such that $x'$ and $y'$ are obtained from $x$ and $y$, respectively, by replacing $b$ with $a'_{1h}$, and $a_{st}$ with $a'_{st}$ for all $s,t$. 
For $2\leq i\leq n$, each $l'_{i}$ is obtained from $l_{i}$ by replacing $b$ with $a'_{1h}$, and $a_{st}$ with $a'_{st}$ for all $s,t$. 
To complete the proof, we use the following.

\begin{claim}\label{claim-sv}
$(1)$~
$\eta_{q}^{(k)}(D,\mathbf{p})(a_{1h})\equiv\eta_{q}^{(k)}(D,\mathbf{p})(b) \pmod{A_{q}^{(k)}R^{(k)}}.$

$(2)$~ 
${\eta}_{q}^{(k)}(D,\mathbf{p})(a_{ij})\equiv\eta_{q}^{(k)}(D',\mathbf{p}')(a'_{ij}) \pmod{A_{q}^{(k)}R^{(k)}}$ 
for any $i, j$.
\end{claim}

Before showing this claim, we observe it implies Proposition~\ref{prop-sv-eta}. 

If $k\neq1$, then the congruence $\eta_{q}^{(k)}(D,\mathbf{p})(l_{k})\equiv\eta_{q}^{(k)}(D',\mathbf{p}')(l'_{k})\pmod{A_{q}^{(k)}R^{(k)}}$ follows from Claim~\ref{claim-sv} directly. 
If $k=1$, then by Lemma~\ref{lem-eta}(1) we have 
\[
\eta_{q}^{(1)}(D,\mathbf{p})(l_{1})=\eta_{q}^{(1)}(D,\mathbf{p})(xy) \hspace{1em} 
\mbox{and}\hspace{1em}  
\eta_{q}^{(1)}(D',\mathbf{p}')(l'_{1})=\eta_{q}^{(1)}(D',\mathbf{p}')(x'y'). 
\] 
Hence Claim~\ref{claim-sv} completes the proof.
\end{proof}

\begin{proof}[{Proof of Claim~\ref{claim-sv}}] 
(1)~
By Lemma~\ref{lem-rho}, it is enough to show that 
\[
\eta_{q}(a_{1h})\equiv\eta_{q}(b)\pmod{A_{q}R}. 
\] 
Lemma~\ref{lem-6k}(2) implies that 
\[
\eta_{q}(b^{-1}a_{1l}^{-\e}a_{1h}a_{1l}^{\e})\equiv0\pmod{A_{q}}.
\] 
Hence it follows that
\begin{eqnarray*}
\eta_{q}(b^{-1}a_{1h})
&=&\eta_{q}(b^{-1}a_{1l}^{-\e}a_{1h}a_{1l}^{\e}a_{1l}^{-\e}a_{1h}^{-1}a_{1l}^{\e}a_{1h}) \\
&\equiv& \eta_{q}([a_{1l}^{-\e},a_{1h}^{-1}]) \pmod{A_{q}} \\
&\equiv& 1 \pmod{R}. 
\end{eqnarray*}

(2)~
Put $\eta'_{q}=\eta_{q}(D',\mathbf{p}')$ for short. 
By Lemma~\ref{lem-rho}, it is enough to show that 
\[
\eta_{q}(a_{ij})\equiv\eta'_{q}(a'_{ij})\pmod{A_{q}R}. 
\] 
for any $1\leq i\leq n$ and $1\leq j\leq m_i+1$. 

This is proved by induction on $q$. 
The assertion is certainly true for $q=1$ or $j=1$. 
Assume that $q\geq1$ and $j\geq2$. 
Let $v_{ij-1}$ and $v'_{ij-1}$ be partial longitudes of $(D,\mathbf{p})$ and $(D',\mathbf{p}')$, respectively. 
If $i\neq1$ or $v_{ij-1}$ does not contain the word $xa_{1l}^{\e}$, then 
$v'_{ij-1}$ is obtained from $v_{ij-1}$ by replacing $b$ with $a'_{1h}$, and $a_{st}$ with $a'_{st}$ for all $s,t$. 
By the proof of (1) and the induction hypothesis, we have 
$\eta_{q}(v_{ij-1})\equiv\eta'_{q}(v'_{ij-1})\pmod{A_{q}R}$. 
Hence Lemma~\ref{lem-p290}(2) implies that 
\[
\eta_{q+1}(a_{ij})\equiv\eta'_{q+1}(a'_{ij})\pmod{A_{q+1}R}. 
\] 
If $i=1$ and $v_{ij-1}$ contains $xa_{1l}^{\e}$, then 
$v_{1j-1}$ and $v'_{1j-1}$ can be written in the forms 
\[
v_{1j-1}=xa_{1l}^{\e}z\hspace{1em} \mbox{and}\hspace{1em}  v'_{1j-1}=x'z'
\] 
for certain words $z\in\overline{A}$ and $z'\in\overline{A'}$ such that $z'$ is obtained from $z$ by replacing $b$ with $a'_{1h}$, and $a_{st}$ with $a'_{st}$ for all $s,t$. 
Then we have
\[
\eta_{q}(xz)\equiv\eta'_{q}(x'z')\pmod{A_{q}R}
\] 
by the proof of (1) and the induction hypothesis. 
Therefore it follows that 
\begin{eqnarray*}
\eta_{q+1}(a_{1j})
&=&\eta_{q}(v_{1j-1}^{-1})\alpha_{1}\eta_{q}(v_{1j-1}) \\
&=& \eta_{q}(z^{-1})\eta_q(a_{1l}^{-\e})\eta_{q}(x^{-1})\alpha_{1}\eta_{q}(x)\eta_q(a_{1l}^{\e})\eta_{q}(z) \\ 
&\equiv& \eta_{q}(z^{-1})\eta_q(a_{1l}^{-\e})\eta_q(a_{1l}^{\e})\left(\eta_{q}(x^{-1})\alpha_{1}\eta_{q}(x)\right)\eta_{q}(z) \pmod{R} \\
&\equiv& \eta'_{q}((z')^{-1}(x')^{-1})\alpha_{1}\eta'_{q}(x'z') \pmod{A_{q+1}R} \\
&=& \eta'_{q}((v'_{1j-1})^{-1})\alpha_{1}\eta'_{q}(v'_{1j-1}) \\
&=& \eta'_{q+1}(a'_{1j}). 
\end{eqnarray*} 
\end{proof}

For a sequence $j_{1}\ldots j_{s}i$ $(1\leq s<q)$ of indices in $\{1,\ldots,n\}$,  
we denote by $\mu_{(D,\mathbf{p})}^{(q,k)}(j_{1}\ldots j_{s}i)$ the coefficient of $X_{j_{1}}\cdots X_{j_{s}}$ in 
$E(\eta_{q}^{(k)}(l_{i}))$. 
We remark that by Lemma~\ref{lem-length}, 
\[
\mu_{(D,\mathbf{p})}^{(q,k)}(j_{1}\ldots j_{s}i)=\mu_{(D,\mathbf{p})}^{(q+1,k)}(j_{1}\ldots j_{s}i).
\]  
Furthermore, by Lemma~\ref{lem-eta}(2), if the sequence $j_{1}\ldots j_{s}$ involves the index $k$, then $\mu_{(D,\mathbf{p})}^{(q,k)}(j_{1}\ldots j_{s}i)=0$. 
On the other hand, if $j_{1}\ldots j_{s}$ does not involve $k$, then $\mu_{(D,\mathbf{p})}^{(q,k)}(j_{1}\ldots j_{s}i)=\mu_{(D,\mathbf{p})}^{(q)}(j_{1}\ldots j_{s}i)~(=\mu_{(D,\mathbf{p})}(j_{1}\ldots j_{s}i))$.

\begin{theorem}\label{th-sv}
Let $(D,\mathbf{p})$ and $(D',\mathbf{p}')$ be virtual link diagrams with base point systems. 
If $(D,\mathbf{p})$ and $(D',\mathbf{p}')$ are related by a self-crossing virtualization, then $\mu_{(D,\mathbf{p})}(I)=\mu_{(D',\mathbf{p}')}(I)$ for any non-repeated sequence $I$. 
\end{theorem} 

\begin{proof} 
Let $k$ be the last index of a non-repeated sequence $I$. 
Then we may put $I=Jk$. 
Since $J$ does not involve $k$, we have 
$\mu_{(D,\mathbf{p})}(Jk)=\mu_{(D,\mathbf{p})}^{(q,k)}(Jk)$  and
$\mu_{(D',\mathbf{p}')}(Jk)=\mu_{(D',\mathbf{p}')}^{(q,k)}(Jk)$.  
To complete the proof, we will show that 
$\mu^{(q,k)}_{(D,\mathbf{p})}(Jk)=\mu^{(q,k)}_{(D',\mathbf{p}')}(Jk)$. 

For $x\in A_{q}^{(k)}R^{(k)}$, we put 
\[
E(x)=1+\sum\nu(j_{1}\ldots j_{s})X_{j_{1}}\cdots X_{j_{s}}. 
\]
By Proposition~\ref{prop-sv-eta}, it is enough to show that 
$\nu(j_{1}\ldots j_{s})=0$ for any non-repeated sequence $j_{1}\ldots j_{s}$ with $s<q$. 

If $x\in A_{q}^{(k)}$, then we have $\nu(j_{1}\ldots j_{s})=0$ by Lemma~\ref{lem-Magnus}. 
If $x\in R^{(k)}$, then we only need to consider the case $x=[\alpha_{i},g^{-1}\alpha_{i}g]$ $(g\in A^{(k)}, 1\leq i\neq k\leq n)$. 
Then it follows that
\begin{eqnarray*}
E(x)-1&=&E([\alpha_{i},g^{-1}\alpha_{i}g])-1 \\
&=&\left(E(\alpha_{i}g^{-1}\alpha_{i}g)-
E(g^{-1}\alpha_{i}g\alpha_{i})\right)E(\alpha_{i}^{-1}g^{-1}\alpha_{i}^{-1}g). 
\end{eqnarray*}
Here we observe that 
\begin{eqnarray*}
&&E(\alpha_{i}g^{-1}\alpha_{i}g)
-E(g^{-1}\alpha_{i}g\alpha_{i}) \\
&&= (1+X_{i})E(g^{-1})(1+X_{i})E(g)
-E(g^{-1})(1+X_{i})E(g)(1+X_{i}) \\
&&= X_{i}E(g^{-1})X_{i}E(g)-E(g^{-1})X_{i}E(g)X_{i}. 
\end{eqnarray*} 
This implies that each term of $E(x)-1$ contains $X_{i}$ at least twice. 
Hence we have $\nu(j_{1}\ldots j_{s})=0$ for any non-repeated sequence $j_{1}\ldots j_{s}$. 
\end{proof}

\begin{proof}[{Proof of Theorem~\ref{th-sv-link}}] 
Let $\mathbf{p}$ and $\mathbf{p}'$ be base point systems of $D$ and $D'$, respectively. 
Then $(D,\mathbf{p})$ and $(D',\mathbf{p}')$ are related by a finite  sequence of self-crossing virtualizations, $\overline{\mbox{w}}$-isotopies and base-change moves. 
If $(D,\mathbf{p})$ and $(D',\mathbf{p}')$ are related by a self-crossing virtualization, 
then by Theorem~\ref{th-sv} 
$\mu_{(D,\mathbf{p})}(I)=\mu_{(D',\mathbf{p}')}(I)$ for any non-repeated sequence $I$. 
This implies that $\Delta_{(D,\mathbf{p})}(I)=\Delta_{(D',\mathbf{p}')}(I)$. 
If $(D,\mathbf{p})$ and $(D',\mathbf{p}')$ are related by a $\overline{\mbox{w}}$-isotopy or base-change moves, 
then it follows from Theorems~\ref{th-w-bar} and~\ref{th-w-iso} that 
$\mu_{(D,\mathbf{p})}(I)\equiv\mu_{(D',\mathbf{p}')}(I)\pmod{\Delta_{(D,\mathbf{p})}(I)}$ and $\Delta_{(D,\mathbf{p})}(I)=\Delta_{(D',\mathbf{p}')}(I)$. 
This completes the proof. 
\end{proof}

\begin{remark}\label{rem-DK-Delta}
It is suggested in \cite{Kotorii,ABMW} that Dye and Kauffman in~\cite{DK} failed to define Milnor-type ``invariants''. 
We clarify why Dye and Kauffman's construction/definition is incorrect. 
In~\cite{DK}, Dye and Kauffman defined a residue class $\omu^{{\rm DK}}$ of Milnor numbers $\mu$ for virtual link diagrams with base point systems. 
Their construction follows Milnor's original work~\cite{M57} 
but a different indeterminacy $\Delta^{{\rm DK}}(j_{1}\ldots j_{r}i)$, 
which is defined as the greatest common divisor of all $\mu(k_{1}\ldots k_{s}i)$, 
where $k_{1}\ldots k_{s}$ is a proper ``subset''  
of $j_{1}\ldots j_{r}$, see \cite[page 945]{DK}. 
(Here, ``subset'' should rather be ``subsequence''.) 
We stress that $\Delta^{{\rm DK}}(j_{1}\ldots j_{r}i)$ is determined by 
Milnor numbers for sequences with the last index $i$. 
It is stated in~\cite[Section~4]{DK} that $\omu^{{\rm DK}}$ does not depend on the choice of base point system, and moreover that it is an invariant of virtual links. 
However, this is {\em wrong}. More precisely, $\omu^{{\rm DK}}$ is {\em not} well-defined even for classical link diagrams. 
In the following, we will show that $\omu^{{\rm DK}}$ does depend 
on both Reidemeister moves and 
the choice of base point system: 
Let $(D, \mathbf{p})$, $(D,\mathbf{p}')$ and $(D',\mathbf{p})$ be 
the $3$-component link diagrams as in Figure~\ref{rem-counter-ex}. 
(We remark that the definition of arcs of a diagram in~\cite{DK} 
coincides with the original one in \cite{M57}.) 
Note that $(D, \mathbf{p})$ and $(D,\mathbf{p}')$ have the same diagram and 
different base point systems, and that $(D, \mathbf{p})$ and $(D',\mathbf{p})$ 
are related by a single R1 move relative base point system. 
Let $l, l'$ and $l''$ be the 3rd longitudes of 
$(D, \mathbf{p})$, $(D,\mathbf{p}')$ and $(D',\mathbf{p})$, respectively. 
Then by the definition of $\eta_{q}$ in~\cite{M57,DK}, 
$\eta_3(l)=\alpha_{2}^{-1}\alpha_{1}^{-1}\alpha_{2}\alpha_{1}$, 
$\eta_3(l')=\eta_3(l'')=1$, and hence 
$E(\eta_{3}(l))=
1+X_2X_1-X_1X_2+{({\rm terms~of~degree~}\geq 3)}$ and   
$E(\eta_3(l'))=E(\eta_3(l''))=1$. 
Since $\Delta^{{\rm DK}}_{(D,\mathbf{p})}(123)=\gcd{\left(
\mu_{(D,\mathbf{p})}(13), \mu_{(D,\mathbf{p})}(23)
\right)}=0$, we have 
$\omu^{{\rm DK}}_{(D,\mathbf{p})}(123)=-1$, while 
$\omu^{{\rm DK}}_{(D,\mathbf{p}')}(123)=
\omu^{{\rm DK}}_{(D',\mathbf{p})}(123)=0$. 
\end{remark}

\begin{figure}[htb]
  \begin{center}
    \vspace{1em}
    \begin{overpic}[width=11cm]{counter-ex.pdf}
      \put(45,-15){$(D,\mathbf{p})$}
      \put(-12,24){$p_{1}$}
      \put(86,10){$p_{2}$}
      \put(122,24){$p_{3}$}
      \put(4,53){$a_{11}$}
      \put(18,24){$a_{21}$}
      \put(52,55){$a_{22}$}
      \put(102,-4){$a_{31}$}
      \put(54,8){$a_{32}$}
      \put(239,-15){$(D,\mathbf{p}')$}
      \put(184,24){$p'_{1}$}
      \put(251,60){$p'_{2}$}
      \put(317,24){$p'_{3}$}
      \put(199,53){$a_{11}$}
      \put(280,38){$a_{21}$}
      \put(280,11){$a_{22}$}
      \put(298,-4){$a_{31}$}
      \put(250,8){$a_{32}$}
    \end{overpic}
  \end{center}
  \vspace{3em}

  \begin{center}
    \vspace{1em}
    \begin{overpic}[width=4.2cm]{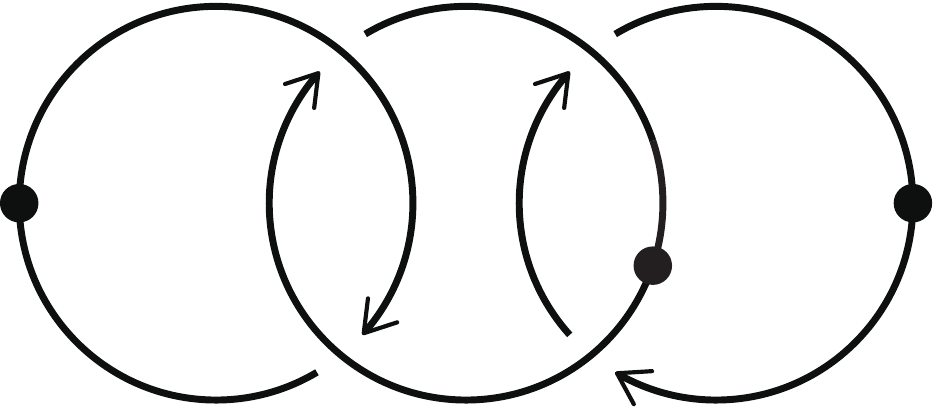}
      \put(46,-15){$(D',\mathbf{p})$}
      \put(-12,25){$p_{1}$}
      \put(88,10){$p_{2}$}
      \put(123,25){$p_{3}$}
      \put(3,54){$a_{11}$}
      \put(18,25){$a_{21}$}
      \put(103,-4){$a_{31}$}
      \put(55,8){$a_{32}$}
    \end{overpic}
  \end{center}
  \vspace{1em}
  \caption{} 
  \label{rem-counter-ex}
\end{figure}

\begin{remark}\label{rem-Th3} 
In Remark~\ref{rem-DK-Delta}, for the original definition of arcs in \cite{M57}, we see that 
$\mu_{(D,\mathbf{p})}(123)\neq\mu_{(D',\mathbf{p})}(123)$, 
while $(D, \mathbf{p})$ and $(D',\mathbf{p})$ 
are related by a single R1 move relative base point system. 
This implies that Theorem~\ref{th-w-bar} does not hold for the original definition  
of arcs.
\end{remark}

\section{Welded string links}
In the previous sections, we have studied Milnor invariants of welded {\em links}. 
Now we address the case of welded {\em string links}. 

Fix $n$ distinct points $0<x_{1}<\cdots<x_{n}<1$ in the unit interval $[0,1]$. 
Let $[0,1]_{1},\ldots,[0,1]_{n}$ be $n$ copies of $[0,1]$. 
An {\em $n$-component virtual string link diagram} is the image of an immersion 
\[
\bigsqcup_{i=1}^{n}[0,1]_{i}\longrightarrow [0,1]\times [0,1]
\]
such that the image of each $[0,1]_{i}$ runs from $(x_{i},0)$ to $(x_{i},1)$, and the singularities are only classical and virtual crossings. 
The $n$-component virtual string link diagram $\{x_{1},\ldots,x_{n}\}\times[0,1]$ in $[0,1]\times[0,1]$ is called the {\em trivial $n$-component string link diagram}. 
An {\em $n$-component welded string link} is an equivalence class of $n$-component virtual string link diagrams under welded isotopy. 

Let $\pi:[0,1]\times [0,1]\rightarrow [0,1]$ be the projection onto the first coordinate. 
Given an $n$-component virtual string link diagram $S$, 
an $n$-component virtual link diagram with a base point system is uniquely obtained by identifying points on the boundary of $[0,1]\times [0,1]$ with their images under the projection $\pi$. 
We denote it by $(D_{S},\mathbf{p}_{S})$, where $\mathbf{p}_{S}=(\pi(x_{1},0),\ldots,\pi(x_{n},0))=(\pi(x_{1},1),\ldots,\pi(x_{n},1))$. 
We see that if two virtual string link diagrams $S$ and $S'$ are welded isotopic, then $(D_{S},\mathbf{p}_{S})$ and $(D_{S'},\mathbf{p}_{S'})$ are $\overline{\mbox{w}}$-isotopic. 

For a sequence $I$ of indices in $\{1,\ldots,n\}$, the {\em Milnor number $\mu_{S}(I)$} of~$S$ is defined to be $\mu_{(D_{S},\mathbf{p}_{S})}(I)$. 
Theorem~\ref{th-w-bar} implies the following directly. 

\begin{corollary}\label{cor-Milnor-stringlink} 
Let $S$ and $S'$ be virtual diagrams of a welded string link. 
Then $\mu_{S}(I)=\mu_{S'}(I)$ for any sequence $I$. 
\end{corollary} 

Combining Theorems~\ref{th-w-bar} and \ref{th-sv}, the following result is obtained immediately. 

\begin{corollary}[{\cite[Lemma 9.1]{MY}}]
\label{cor-sv}
If two virtual string link diagrams $S$ and $S'$ are related by a finite sequence of self-crossing virtualizations and 
welded isotopies, then $\mu_{S}(I)=\mu_{S'}(I)$ 
for any non-repeated sequence $I$.
\end{corollary}

\begin{remark}\label{rem-sv}
The converse of Corollary~\ref{cor-sv} is also true. 
In fact, it is shown in~\cite{ABMW,MY} that 
Milnor numbers for non-repeated sequences classify virtual string link diagrams up to self-crossing virtualizations and welded isotopies. 
\end{remark}

We conclude this paper with a classification result of virtual link diagrams with base point systems up to an equivalence relation generated by self-crossing virtualizations and $\overline{\mbox{w}}$-isotopies. 

\begin{theorem}\label{sv-classification}
Let $(D,\mathbf{p})$ and $(D',\mathbf{p}')$ be virtual link diagrams with base point systems. 
Then the following are equivalent. 
\begin{enumerate}
\item $(D,\mathbf{p})$ and $(D',\mathbf{p}')$ are related by a finite sequence of self-crossing virtualizations and $\overline{\mbox{w}}$-isotopies. 
\item $\mu_{(D,\mathbf{p})}(I)=\mu_{(D',\mathbf{p}')}(I)$ for any non-repeated sequence $I$. 
\end{enumerate}
\end{theorem}

\begin{proof}
$(1)\Rightarrow(2)$:~
This follows from Theorems~\ref{th-w-bar} and \ref{th-sv} directly. 

$(2)\Rightarrow(1)$:~
For a small disk $\delta$ which is disjoint from $(D,\mathbf{p})$ (or $(D',\mathbf{p}')$), by applying VR2 relative base point system and the local move in Figure~\ref{w-bar} repeatedly, 
we can deform $(D,\mathbf{p})$ (or $(D',\mathbf{p}')$) such that the intersection between the disk $\delta$ and the deformed diagram is the trivial string link diagram whose each component contains the base point. 
Hence, $D\setminus \delta$ and $D'\setminus \delta$ can be regarded as string link diagrams $S$ and $S'$, respectively. 
Since $(D_{S},\mathbf{p}_{S})$ and $(D_{S'},\mathbf{p}_{S'})$ are $\overline{\mbox{w}}$-isotopic to $(D,\mathbf{p})$ and $(D',\mathbf{p}')$, respectively, 
it follows from Theorem~\ref{th-w-bar} that 
\[
\mu_{S}(I)=\mu_{(D,\mathbf{p})}(I)\hspace{1em} \mbox{and}\hspace{1em}  \mu_{S'}(I)=\mu_{(D',\mathbf{p}')}(I)
\] 
for any non-repeated sequence $I$. 
Hence we have $\mu_{S}(I)=\mu_{S'}(I)$ by assumption. 
Then, by Remark~\ref{rem-sv}, $S$ and $S'$ are related by a finite sequence of self-crossing virtualizations and welded isotopies. 
This implies that $(D_{S},\mathbf{p}_{S})$ and $(D_{S'},\mathbf{p}_{S'})$ are related by a finite sequence of self-crossing virtualizations and $\overline{\mbox{w}}$-isotopies. 
\end{proof}

\begin{remark}
By Theorem~\ref{sv-classification}, the two virtual link diagrams with base point systems $(D, \mathbf{p})$ and $(D, \mathbf{p}')$ given in Example~\ref{ex-Milnor-number} are not related by a finite sequence of self-crossing virtualizations and $\overline{\mbox{w}}$-isotopies.
\end{remark}


\end{document}